\documentclass[%
  onecolumn
]{mpi2015-cscpreprint}

\usepackage{algorithm, algpseudocode, algorithmicx}
\usepackage{amsthm} 
        \theoremstyle{plain}
        \newtheorem{theorem}{Theorem}
	\newtheorem{corollary}{Corollary}
 	\newtheorem{lemma}{Lemma}
	\newtheorem{proposition}{Proposition}

        \theoremstyle{definition}
	
	\newtheorem{example}{Example}

        \theoremstyle{remark}
		
\usepackage{amsmath, amscd, amssymb, bm, mathrsfs, mathtools}
\usepackage{array, booktabs, makecell} 
\newcolumntype{P}[1]{>{\centering\arraybackslash}p{#1}}
\newcolumntype{M}[1]{>{\centering\arraybackslash}m{#1}}
\usepackage[american]{babel}
\usepackage{caption}
\usepackage[ddmmyy,hhmmss]{datetime}
\usepackage{enumerate}
\usepackage{graphicx}
\usepackage{setspace}
\usepackage[backgroundcolor=violet!60,textcolor=white,linecolor=gray]{todonotes}
\usepackage{url}
\usepackage{xspace}

\newcommand{\memmax}{\texttt{mem}_{\max}}

\newcommand{\vc}{\bm{c}}

\newcommand{\ve}{\bm{e}}

\newcommand{\vg}{\bm{g}}
\newcommand{\vgline}{\underline{\vg}}

\newcommand{\vx}{\bm{x}}
\newcommand{\vy}{\bm{y}}

\newcommand{\vC}{\bm{C}}

\newcommand{\vD}{\bm{D}}

\newcommand{\vE}{\bm{E}}
\newcommand{\vEline}{\underline{\vE}}

\newcommand{\vM}{\bm{M}}

\newcommand{\vV}{\bm{V}}

\newcommand{\vX}{\bm{X}}

\newcommand{\vZ}{\bm{Z}}

\newcommand{\bHH}{\bm{\mathsf{H}}}
\newcommand{\bHHline}{\underline{\bHH}}
\newcommand{\bJJ}{\bm{\mathsf{J}}}
\newcommand{\bMM}{\bm{\mathsf{M}}}

\newcommand{\bVV}{\bm{\mathcal{V}}}

\newcommand{\DD}{\mathcal{D}}

\newcommand{\LL}{\mathcal{L}}
\newcommand{\HH}{\mathcal{H}}
\newcommand{\HHline}{\underline{\HH}}
\newcommand{\II}{\mathcal{I}}
\newcommand{\IIline}{\underline{\II}}
\newcommand{\JJ}{\mathcal{J}}

\renewcommand{\LL}{\mathcal{L}}
\newcommand{\MM}{\mathcal{M}}

\newcommand{\QQ}{\mathcal{Q}}
\newcommand{\QQbar}{\bar{\mathcal{Q}}}

\newcommand{\spK}{\mathscr{K}} 
\newcommand{\spP}{\mathbb{P}} 
\newcommand{\spC}{\mathbb{C}}
\newcommand{\spR}{\mathbb{R}}

\DeclareMathOperator{\blkspn}{span^\square}

\DeclareMathOperator{\vect}{vec}

\DeclareMathOperator{\spec}{spec}
\renewcommand{\vec}[1]{\vect\left(#1\right)}

\DeclareMathOperator*{\argmin}{arg\,min}

\newcommand{\eps}{\varepsilon}
\newcommand{\bigO}[1]{\mathcal{O}\left(#1\right)}

\newcommand{\inv}{{-1}}
\newcommand{\cinv}{{-*}}

\newcommand{\pinv}{{\dagger}}

\newcommand{\one}{{(1)}}
\newcommand{\two}{{(2)}}

\renewcommand{\k}{{(k)}}
\newcommand{\kless}{{(k-1)}}

\newcommand{\G}{\mbox{\normalfont\tiny\sffamily G}}
\newcommand{\MR}{\mbox{\normalfont\tiny\sffamily MR}}
\newcommand{\PMR}{\mbox{\normalfont\tiny\sffamily PMR}}
\newcommand{\NKS}{\mbox{\normalfont\tiny\sffamily NKS}}
\renewcommand{\mod}{\mbox{\normalfont\tiny\sffamily mod}}

\newcommand{\norm}[1]{\left\lVert#1\right\rVert}
\newcommand{\normF}[1]{\left\lVert#1\right\rVert_{\text{F}}}
\newcommand{\abs}[1]{\left|#1\right|}

\newcommand{\kl}[1]{{#1}}
\newcommand{\dpp}[1]{{#1}}

\newcommand{\badconddiag}{\texttt{bad\_cond\_diag}\xspace}
\newcommand{\convdiff}{\texttt{conv\_diff\_3d}\xspace}
\newcommand{\laplacian}{\texttt{laplacian\_2d}\xspace}
\newcommand{\logdiag}{\texttt{log\_diag}\xspace}
\newcommand{\rail}{\texttt{rail\_1357}\xspace}
\newcommand{\iss}{\texttt{iss}\xspace}
\DeclareMathOperator{\diag}{diag} 

\begin{document}


\title{Low-rank-modified Galerkin methods for the Lyapunov equation}
  
\author[$\ast$]{Kathryn Lund}
\affil[$\ast$]{Computational Methods in Systems and Control Theory, Max Planck Institute for Dynamics of Complex Technical Systems, Sandtorstr.\ 1, Magdeburg, 39106, Germany.\authorcr
  \email{lund@mpi-magdeburg.mpg.de}, \orcid{0000-0001-9851-6061}}
  
\author[$\dagger$]{Davide Palitta}
\affil[$\dagger$]{Dipartimento di Matematica, Piazza di Porta S. Donato 5, 40126 Bologna, Italy\authorcr
  \email{davide.palitta@unibo.it}, \orcid{0000-0002-6987-4430}}
  
\shorttitle{}
\shortauthor{K.~Lund, D.~Palitta}
\shortdate{}
  

  
\abstract{Of all the possible projection methods for solving large-scale Lyapunov matrix equations, Galerkin approaches remain much more popular than minimal-residual ones. This is mainly due to the different nature of the projected problems stemming from these two families of methods. While a Galerkin approach leads to the solution of a low-dimensional matrix equation per iteration, a matrix least-squares problem needs to be solved per iteration in a minimal-residual setting. The significant computational cost of these least-squares problems has steered researchers towards Galerkin methods in spite of the appealing properties of minimal-residual schemes. In this paper we introduce a framework that allows for modifying the Galerkin approach by low-rank, additive corrections to the projected matrix equation problem with the two-fold goal of attaining monotonic convergence rates similar to those of minimal-residual schemes while maintaining essentially the same computational cost of the original Galerkin method. We analyze the well-posedness of our framework and determine possible scenarios where we expect the residual norm attained by two low-rank-modified variants to behave similarly to the one computed by a minimal-residual technique. A panel of diverse numerical examples shows the behavior and potential of our new approach.}

\novelty{A new framework for projection methods is developed for Lyapunov matrix equations.  In particular, this framework permits a cheap estimate to minimal-residual methods using low-rank modifications and block Krylov subspace methods.}

\maketitle
\section{Introduction} \label{sec:intro}
We are interested in the numerical solution of large-scale Lyapunov equations of the form
\begin{equation} \label{eq:lyap}
	A X + X A^* + \vC \vC^* = 0,
\end{equation}
where $A \in \spC^{n \times n}$ is large and sparse and $\vC \in \spC^{n \times r}$, $r \ll n$, has low rank.  Throughout the text, we use plain, uppercase letters to refer to square matrices (e.g., $A$), boldface uppercase letters to refer to \emph{block vectors} (e.g., $\vC$), and boldface lowercase letters for column vectors (e.g., $\vc$).  When a square matrix has block structure, it becomes italicized ($\HH$), and the concatenation of block vectors into a basis is denoted with bold italics (e.g., $\bVV$). Matrices with Kronecker structure are formatted with bold sans serif font (e.g., $\bHH$).

The Lyapunov equation \eqref{eq:lyap} is encountered in many applications. For instance, in some model reduction~\cite{Ant05b} and robust/optimal control strategies \cite{ZhoDG96} a Lyapunov equation, or a sequence of such equations, has to be solved. Moreover, the discretization of certain elliptic partial differential equations (PDEs) leads to an algebraic problem that can be often represented in terms of a Lyapunov equation; see, e.g., \cite{PalS16}. We refer the interested reader to the survey papers \cite{BenS13a,Sim16a} and the references therein for more details about the aforementioned applications and further research areas where Lyapunov equations play an important role.

We assume the matrix $A$ to be stable: its spectrum is contained in the left-half open complex plane $\spC_-$. Therefore, the solution $X$ to \eqref{eq:lyap} is Hermitian positive semi-definite \cite{SnyZ70}. Moreover, it is well-known that the singular values of $X$ rapidly decay to zero, if certain further assumptions on $A$ are considered; see, e.g., \cite{BakES15,Pen00}. In this case, the solution $X$ can be well approximated by a low-rank matrix $\vZ \vZ^*\approx X$, $\vZ \in \spC^{n \times t}$, $t \ll n$, and the computation of the low-rank factor $\vZ$ is the task of the so-called \emph{low-rank methods}. Numerous, diverse algorithms belong to this broad class of solvers. Some examples are projection methods \cite{DruS11,Sim07}, low-rank Alternating Direction Implicit (ADI) methods \cite{LiW02,Kur16}, sign function methods \cite{BauB06,Bau08a}, and Riemannian optimization methods \cite{VanV10}. A more complete list of low-rank solvers can be found in the surveys \cite{BenS13a, Sim16a}.

We focus on projection methods, and we propose a novel framework for solving \eqref{eq:lyap}. Given a suitable subspace $\spK$ and a matrix $\bVV_m \in \spC^{n \times mr}$, $mr < n$, whose columns represent an orthonormal basis of $\spK$, projection methods seek an approximate solution $X_m$ to \eqref{eq:lyap} of the form $X_m = \bVV_m Y_m \bVV_m^*$. The square matrix $Y_m \in \spC^{mr \times mr}$ can be computed in different ways. Most of the schemes available in the literature compute $Y_m$ by imposing a Galerkin condition on the residual matrix $R_m = A X_m + X_m A^*-\vC \vC^*$. Imposing such a condition, namely $\bVV_m^* R_m \bVV_m = 0$, is equivalent to computing $Y_m$ as the solution of the \emph{projected} Lyapunov equation
\begin{equation} \label{eq:lyap_projected}
	\HH_m Y + Y \HH_m^* + \vE_1 \Gamma \Gamma^* \vE_1^* = 0,
\end{equation}
where $\HH_m = \bVV_m^* A \bVV_m$, $\vV_1 \Gamma$ is the economic QR factorization of $\vC$, and $\vE_1=e_1\otimes I_r$; see, e.g., \cite{Sim07, DruS11}.

A less explored alternative consists of computing $Y_m$ by imposing a \dpp{minimal residual (MR) condition, namely we compute} $Y_m$ as follows:
\begin{equation} \label{eq:lyap_projected_MR}
	Y_m = \argmin_Y \normF{ \HHline_m Y [I_m \,\, 0] + [I_m \,\, 0]^* Y \HHline_m^* + \vEline_1 \Gamma \Gamma^* \vEline_1 },
\end{equation}
where $\HHline_m = \bVV_{m+1}^* A \bVV_m$ and $\normF{\cdot}$ is the Frobenius norm\footnote{Computing $Y_m$ as in \eqref{eq:lyap_projected_MR} is equivalent to imposing a Petrov-Galerkin condition on the residual for a specific choice of the test space; see, e.g.,~\cite[Section 5]{PalS20}.}; see, e.g., \cite{LinS13, HuR92}.

In spite of their appealing minimization property, MR methods for matrix equations are not commonly adopted.  The solution of the matrix least squares problem in \eqref{eq:lyap_projected_MR} can be remarkably more expensive than solving \eqref{eq:lyap_projected}; see \cite{LinS13}, as well as our own numerical results in Section~\ref{Numerical results}.  Moreover, the matrix $Y_m$ computed by \eqref{eq:lyap_projected_MR} may be indefinite (in floating-point as well as in exact arithmetic), meaning that the computed approximation $X_m = \bVV_m Y_m \bVV_m^*$ is as well, even though the solution $X$ to \eqref{eq:lyap} is semi-definite.  For further discussion about this peculiar drawback of MR methods for Lyapunov equations, see \cite[Section~5]{PalS20}, and for a similar approach for algebraic Riccati equations with global Krylov subspace methods, see \cite{JbiR19}.

Instead, we propose to compute $Y_m$ as the solution of a low-rank modification (LRM) of \eqref{eq:lyap_projected}, whereby the coefficient matrix $\HH_m$ is replaced by $\HH_m + \MM$ for a certain low-rank matrix $\MM$.  This approach is inspired by a similar technique for univariate matrix functions~\cite{FroLS20}. In particular, we consider the so-called ``harmonic modification'' $\MM$, which addresses some computational issues of projection methods while maintaining interesting theoretical features.  We refer to the resulting method as a pseudo-minimal residual (PMR) method, given that it appears to closely approximate the MR approximation for \eqref{eq:lyap} with symmetric $A$.

In Section~\ref{sec:lrm_framework} we introduce a bivariate LRM framework that generalizes the univariate one from~\cite{FroLS20}.  We precisely define the PMR method in Section~\ref{sec:pmr}, quantify how close it is to MR via an eigenvalue analysis, and propose another LRM that minimizes the residual over a structured space of Kronecker sums.  We show how general LRM approaches can be combined with the compress-and-restart strategy of \cite{KreLMetal21} in Section~\ref{sec:cr}.  Results of numerical experiments are presented in Section~\ref{sec:num_ex}, and we summarize our findings and contributions in Section~\ref{sec:conclusions}.

\section{Low-rank modification framework} \label{sec:lrm_framework}
An LRM framework for block Krylov subspace methods (KSMs) has been previously developed for matrix functions; see, in particular, \cite{FroLS20}.  While much of the framework transfers easily to matrix equations, we must take some care, since we are moving from a univariate framework to a bivariate one.

We begin by defining the $m$th \emph{block Krylov subspace} for $A$ and $\vC$ in terms of the block span, denoted here as $\blkspn$:
\begin{align}
	\spK_m(A, \vC)
	& := \blkspn\{\vC, A \vC, \ldots, A^{m-1} \vC\} = \left\{\sum_{j=0}^{m-1} A^{j} \vC \Gamma_j : \{\Gamma_j\}_{j=0}^{d} \subset \spC^{r \times r} \right\}. \label{eq:block_Krylov_subspace}
\end{align}

Note that $\spK_m(A, \vC) \subset \spC^{n \times r}$, i.e., its elements are \emph{block vectors}.  This is in contrast to the column span treatment of block KSMs, which is often used for solving linear systems with a column right-hand side.  For a foundational resource on the different interpretations of block KSMs and how they relate to each other, see \cite{Gut07}.

Generating $\spK_m(A, \vC)$ via the block Arnoldi process gives rise to the \emph{block Arnoldi relation}
\begin{equation} \label{eq:arnoldi}
	A \bVV_m = \bVV_{m+1} \HHline_m = \bVV_m \HH_m + \vV_{m+1} H_{m+1,m} \vE_m^*,
\end{equation}
where $\bVV_{m+1} \in \spC^{n \times (m+1)r}$ is orthonormal, $\HHline_m = \begin{bmatrix} \HH_m \\ H_{m+1,m} \vE_m^* \end{bmatrix}$ is block upper Hessenberg, and $\vE_m = \ve_m \otimes I_r$ is a unit block vector.

With the block Arnoldi decomposition, we can directly write the Galerkin approximation to \eqref{eq:lyap} over $\spK_m(A, \vC)$.  First, we project \eqref{eq:lyap} down and solve the $mr \times mr$ problem
\begin{equation} \label{eq:lyap_galerkin}
	\HH_m Y + Y \HH_m^* + \vE_1 \Gamma \Gamma^* \vE_1^* = 0.
\end{equation}
Letting $Y_m^{\G} \in \spC^{mr \times mr}$ denote the solution to~\eqref{eq:lyap_galerkin}, we define the Galerkin approximation to~\eqref{eq:lyap} as
\begin{equation} \label{eq:X_galerkin}
	X_m^{\G} := \bVV_m Y_m^{\G} \bVV_m^*.
\end{equation}

We can make LRMs to the Galerkin solution in much the same way as for matrix functions by simply replacing $\HH_m$ in \eqref{eq:lyap_galerkin} with $\HH_m + \MM$, where $\MM = \vM \vE_m^*$ for some matrix $\vM \in \spC^{mr \times r}$.  It is possible to choose $\MM$ so that $\HH_m + \MM$ has harmonic Ritz values (as in GMRES) or so that certain Ritz values are prescribed; see, e.g., \cite{FroGS14a, FroLS17, FroLS20} and Section~\ref{sec:pmr}.  Most importantly, it is possible to choose $\MM$ so that the resulting approximation is still in the product of the Krylov subspaces $\spK_m(A, \vC) \times \spK_m(A, \vC)$, which Lemma~\ref{lem:poly_exact} demonstrates.

Letting $Y_m^{\mod}$ denote the solution to the modified problem,
\begin{equation} \label{eq:lyap_mod_galerkin}
	(\HH_m + \MM) Y + Y (\HH_m + \MM)^* + \vE_1 \Gamma \Gamma^* \vE_1^* = 0,
\end{equation}
leads to an approximate solution $X_m^{\mod} := \bVV_m Y_m^{\mod} \bVV_m^*$, whose residual norm can be cheaply computed, as shown in the next proposition.
\begin{proposition} \label{prop:res_lrm_galerkin}
    Let $Y_m^{\mod}$ be the solution to the Lyapunov equation \eqref{eq:lyap_mod_galerkin}. Then the residual matrix $R_m^{\mod} = A X_m^{\mod} + X_m^{\mod} A^* + \vC \vC^*$ can be written as
    \begin{equation} \label{eq:res_comp_mod}
        R_m^{\mod} = \bVV_{m+1} G_m
        \begin{bmatrix}
                & I_r 	&		\\
            I_r & 		& -I_r	\\
                & -I_r 	&
        \end{bmatrix}
        G_m^* \bVV_{m+1}^* \in \spC^{n \times n},
    \end{equation}
    where
    \[
    G_m =
        \begin{bmatrix}
        \vE_{m+1}H_{m+1,m}	& \IIline Y_m^{\mod} \vE_m	& \IIline \vM
    \end{bmatrix}
    \in \spC^{n \times 3r}
    \]
    and
    $\IIline = \begin{bmatrix} I_{mr} \\ 0 \end{bmatrix} \in \spR^{(m+1)r \times mr}$. Moreover,
    \begin{equation} \label{eq:res_norm_mod}
        \|R_m^{\mod}\|_F^2 = 2\left(\|Y_m^{\mod}\vE_mH_{m+1,m}^*\|_F^2+
        \|Y_m^{\mod}\vE_m\vM^*\|_F^2+2\cdot\text{trace}((\vE_m^*Y_m^{\mod}\vM
        )^2)\right).
    \end{equation}
	
\end{proposition}
\begin{proof}
	Thanks to the block Arnoldi relation~\eqref{eq:arnoldi} and since  $X_m^{\mod} := \bVV_m Y_m^{\mod} \bVV_m^*$,
	\begin{align*}
		R_m^{\mod} = & AX_m^{\mod} + X_m^{\mod} A^* - \vC \vC^*	\\
		= &	\bVV_m \left( (\HH_m +\MM) Y_m^{\mod} + Y_m^{\mod} (\HH_m +\MM)^*
		- \vE_1 \Gamma \Gamma^* \vE_1 \right) \bVV_m^*	\\
		& + \vV_{m+1} H_{m+1,m} \vE_m^*Y_m^{\mod} \bVV_m^*
		+ \bVV_m Y_m^{\mod} \vE_m H_{m+1,m}^* \vV_{m+1}^*	\\
		& -\bVV_m\MM Y_m^{\mod}\bVV_m^* - \bVV_m Y_m^{\mod}\MM^*\bVV_m^*	\\
		= & \bVV_{m+1} \left( \vE_{m+1} H_{m+1,m} \vE_m^* Y_m^{\mod} \IIline^* + \IIline Y_m^{\mod} \vE_m H_{m+1,m}^* \vE_{m+1}^* - \IIline \MM Y_m^{\mod} \IIline^* \right.	\\
		& \left.- \IIline Y_m^{\mod} \MM^*\IIline^* \right) \bVV_{m+1}^*.
	\end{align*}
	The results follow by plugging the low-rank form $\MM = \vM \vE_m^*$ in the expression above. Similarly, by recalling that $\vE_{m+1}^*\IIline=0$, $\IIline^*\IIline=I_{mr}$, and $\vE_{m+1}^*\vE_{m+1}=I_r$, a direct computation shows~\eqref{eq:res_norm_mod}.
\end{proof}

In Algorithm~\ref{alg:LRMG} the low-rank modified Galerkin approach for~\eqref{eq:lyap} is outlined, where $m_{\max}$ denotes the maximum number of block basis vectors and $\eps > 0$ is the desired relative residual tolerance. Note that the algorithm reduces to the standard Galerkin approach whenever $\MM = 0$.
\begin{algorithm}[t!]
	\caption{LRM Galerkin approach for Lyapunov equations \label{alg:LRMG}}
	\begin{algorithmic}[1]
    	\State{\textbf{input} $A \in \spC^{n\times n}$, $\vC \in \spC^{n \times r}$, $m_{\max}$, $\eps$}
        \State{\textbf{output} $\bVV_m = \begin{bmatrix} \vV_1 & \vV_2 & \cdots & \vV_m \end{bmatrix} \in \spC^{n\times mr}$, $Y_m^{\mod} \in \spC^{mr \times mr}$}
    	\State{Compute an economic QR factorization of $\vC = \vV_1 \Gamma$}
    	\For{$m = 1, \ldots, m_{\max}$}
    		\State{Compute the next basis block $\vV_{m+1}$ and update $\HH_m$}
    		\State{Compute $Y_m^{\mod}$ as the solution of~\eqref{eq:lyap_mod_galerkin}}
    		\State{Compute $\normF{R_m^{\mod}}$}
    		\If{$\normF{R_m^{\mod}} \leq \eps \cdot \normF{\vC^*\vC}$}
    		    \State{\textbf{return} $\bVV_m$ and $Y_m^{\mod}$}
		    \EndIf
		 \EndFor
    \end{algorithmic}
\end{algorithm}

The connection between matrix equations and bivariate functions of matrices has been well established by Kressner \cite{Kre14, Kre19}.  In fact, the LRM framework holds for bivariate matrix functions in general, and thus a wide array of other applications, such as Fr\'echet derivatives, the Stein equation, and time-limited and frequency-limited balanced truncation model reduction \cite{BenKS16a, GawJ90}.

Let $\spP_{k, \ell}(\spC, \spC)$ denote the space of bivariate matrix polynomials of degree $k$ in the first variable and degree $\ell$ in the second.  Like univariate matrix functions \cite{Hig08a}, bivariate matrix functions can be defined as Hermite interpolating polynomials evaluated on the matrices $A$ and $B$ (of compatible dimension).  In more detail,
\[
	f\{A, B\}(\vC \vD^*) := p_{k,l}\{A, B\}(\vC \vD^*) := \sum_{i=0}^{k} \sum_{j=0}^{\ell} \alpha_{ij} A^i \vC (B^j \vD)^*,
\]
for some degrees $k$ and $\ell$ and scalars $\alpha_{ij}$, as long as assumptions on the eigenvalues of $A$ and $B$ are met; see \cite[Definition~2.3]{Kre14}.  In particular, the solution $X$ to the Lyapunov equation~\eqref{eq:lyap} can be expressed as the evaluation of the function $f(x,y) = \frac{1}{x+y}$, i.e., $X = f\{A, A^*\}(\vC \vC^*)$.

We state an alternative to \cite[Lemma~2]{Kre19} that accounts for LRMs for Lyapunov equations.
\begin{lemma} \label{lem:poly_exact}
    Let $\HH, \bVV$ denote the Arnoldi matrix and basis from \eqref{eq:arnoldi}.  For all bivariate polynomials $q$ of degrees $m-1, m-1$ or less, it holds that
    \[
	q(A,A^*)(\vC \vC^*) = \bVV \cdot q\{\HH + \MM, (\HH + \MM)^*\}(\vE_1 \Gamma \Gamma^* \vE_1) \cdot \bVV^*,
    \]
    as long as $\MM = \vM \vE_m^*$ for some $\vM \in \spC^{n \times r}$.
\end{lemma}

The proof follows from a straightforward combination of \cite[Lemma~2]{Kre19} and \cite[Theorem~2.7]{FroLS20}.  The extension to Sylvester equations is straightforward; one just has to be careful with the dimensions for the second Krylov subspace.

In work being developed in parallel, it has been shown that equations of the form~\eqref{eq:lyap_galerkin} arise in Krylov subspace techniques combined with sketching; see~\cite{PalSS23a}. However, there is a fundamental difference between this setting and our novel low-rank modified Galerkin method. Indeed, while in the former the form of $\vM$ is solely dictated by the use of sketching, here we choose $\vM$ to try to meet a target behaviour in our solver.  
\section{A Pseudo-Minimal Residual Method} \label{sec:pmr}
The Generalized Minimal Residual method (GMRES) is a popular approach for linear systems, precisely because of the guaranteed monotonic behavior of the residual \cite{HesS52}.  It can be easily shown that (block) GMRES is equivalent in exact arithmetic to an LRM of the Full Orthogonalization Method (FOM); see \cite[Theorem~3.3]{SimG96} or \cite{FroGS14a, FroLS20}. By assuming $\HH_m$ is nonsingular and choosing
\begin{equation} \label{eq:vM_PMR}
    \vM^{\PMR} := \HH_m^\cinv \vE_m H_{m+1,m}^* H_{m+1,m},
\end{equation}
and $\MM^{\PMR} := \vM^{\PMR} \vE_m^*$, it holds that
\[
    \bVV_{m+1} \HHline_m^\pinv \vEline_1 = \bVV_{m+1} (\HH_m + \MM^{\PMR})^\inv \vE_1,
\]
which is also the GMRES solution for the linear system $A \vX = \vC$.

The modification $\MM^{\PMR}$ in \eqref{eq:lyap_mod_galerkin} does not guarantee a minimal residual approximation to~\eqref{eq:lyap}, but it comes close in many cases, leading to our moniker ``pseudo-minimal residual" (PMR).  The well-posedness of~\eqref{eq:lyap_mod_galerkin} is difficult to prove in general, since predicting the impact of the low-rank modification $\MM$ on the spectral properties of $\HH_m$ is seldom doable. In the next proposition we show that this is possible in case of $\MM^{\PMR}$.

\begin{proposition}\label{prop_Ysemi-definite}
    Let $A$ have its field of values in the left half of the complex plane, denoted here as $\spC^{-}$.\footnote{This property has been denoted as ``negative definiteness" or ``negative realness" in the literature, but with varying consistency, and we spell it out to avoid confusion.} Then for all $m$ such that $mr < n$ and such that the block Arnoldi algorithm does not break down, $\HH_m + \MM^{\PMR}$ is stable (i.e., its eigenvalues have negative real part), implying in particular that equation~\eqref{eq:lyap_mod_galerkin} has a unique, positive semi-definite solution.
\end{proposition}

\begin{proof}
    The proof is essentially a special case of part of the proof of \cite[Corollary~4.4]{FroLS20}.  For ease of reading, we reproduce it in our setting here.
    
    Fix $m$ such that $mr < n$ and assume that the Arnoldi algorithm does not break down.  It follows from~\eqref{eq:arnoldi} and the assumption on $A$ that $\HH_m$ also has its field of values in $\spC^{-}$. Now let $(\lambda, \vx)$ be an eigenpair of $\HH_m + \MM^{\PMR}$; namely,
    \begin{equation} \label{eq:eigenpair}
        (\HH_m + \MM^{\PMR}) \vx = \lambda \vx.
    \end{equation}
    Recall that $\MM^{\PMR} = \vM^{\PMR} \vE_m^* = \HH_m^\cinv \vE_m H_{m+1,m}^* H_{m+1,m}\vE_m^*$.  Left-multiplying \eqref{eq:eigenpair} by $\HH_m^*$ leads to
    \[
        (\HH_m^*\HH_m + \vE_m H_{m+1,m}^* H_{m+1,m}\vE_m^*) \vx = \lambda \HH_{m}^* \vx.
    \]
    Therefore,
    \begin{equation} \label{eq:lambda}
        \lambda =
        \frac{\vx^* \HH_m^* \HH_m \vx + \vx^* \vE_m H_{m+1,m}^* H_{m+1,m} \vE_m \vx}{\vx^* \HH_m^* \vx}.
    \end{equation}
    Since the numerator in \eqref{eq:lambda} is a positive real number whereas the denominator has negative real part, $\lambda$ has negative real part as well; consequently, $\HH_m + \MM^{\PMR}$ is stable.
    To conclude, the solution of \eqref{eq:lyap_mod_galerkin} for PMR is positive semi-definite, as it is the solution to a Lyapunov equation with a stable coefficient matrix and a positive semi-definite constant term given by $\vE_1 \Gamma \Gamma^* \vE_1^*$.
\end{proof}

We have already mentioned that, in the MR framework, the matrix $Y_m$ is not necessarily semi-definite even in the case of $A$ with field of values in $\spC^{-}$, as the former is the solution to the least squares problem~\eqref{eq:lyap_projected_MR}. The semi-definiteness of $Y_m$ can be included in the formulation of the least squares problem as an additional constraint; see, e.g.,~\cite{PalS20}. However, this constraint increases the already expensive computational cost of the solution of~\eqref{eq:lyap_projected_MR}. Proposition~\ref{prop_Ysemi-definite} shows that the semi-definiteness of $Y_m^{\PMR}$ is guaranteed in our novel setting.

In the following, we explore connections between PMR and actual minimal residual methods.
\subsection{Connections to the minimal residual method}
Recall that in the minimal residual framework, $Y_m$ solves the matrix least squares problem~\eqref{eq:lyap_projected_MR}.  By considering the Kronecker representations of the solutions to \eqref{eq:lyap_projected_MR} and \eqref{eq:lyap_mod_galerkin}, we can more directly relate the two approaches.

Define
\[
    \bHHline_m := \HHline_m \otimes \IIline + \IIline \otimes \HHline_m,
\]
and
\[
    \vgline_m = \vec{\vEline_1 \Gamma \Gamma^* \vEline_1^*}.
\]
The nonsingularity of the Kronecker matrix $\bHH_m$ follows from that of $\HH_m$.  Note that we work with these quantities only theoretically, as $\bHHline_m$ has dimensions $(m+1)^2 r^2 \times m^2 r^2$ and is dense, making it prohibitively expensive to form for even moderate $m$ and $r$ in practice.  The vector $\vgline_m$ has length $(m+1)^2r^2$.  We can therefore define $\vy_m^{{\MR}} := -\bHHline^\dagger \vgline_m$, which minimizes $\normF{\bHHline_m \vy + \vgline_m}$.  In fact, due to the equivalence between \eqref{eq:lyap_projected_MR} and the Kronecker problem $\normF{\bHHline_m \vy + \vgline_m}$, it holds that $\vy_m^{{\MR}} = \vec{Y_m^{{\MR}}}$; see, e.g., \cite[Chapter~4]{HorJ91}.

The PMR solution can be expressed in vectorized form as $\vy_m^{\PMR} := -\big(\bHH_m + \bMM_m^{\PMR} \big)^\inv \vg_m$, where
\begin{equation} \label{eq:bMM_PMR}
    \bMM_m^{\PMR} := \HH_m^\cinv \JJ_m \otimes \II + \II \otimes \HH_m^\cinv \JJ_m \mbox{ and } \JJ_m := \vE_m H_{m+1,m}^* H_{m+1,m} \vE_m^*.
\end{equation}

Note as well that for $\bMM^{\G} := 0$, we recover the vectorized Galerkin solution $\vy_m^{\G}$.  Furthermore, the following result holds true. 

\begin{proposition} \label{prop:MR_LRM_form}
    Define
    \begin{equation} \label{eq:bMM_MR}
        \bMM_m^{\MR} := \bHH_m^\cinv \big( \JJ_m \otimes \II + \II \otimes \JJ_m \big).
    \end{equation}
    Then
    \begin{equation} \label{eq:MR_LRM_form}
        \vy_m^{\MR} = -(\bHH_m + \bMM_m^{\MR})^\inv \vg_m.
    \end{equation}
\end{proposition}
\begin{proof}
    The Moore-Penrose inverse (and $\bHHline_m$ having full column rank) gives
    \begin{equation} \label{eq:bHHline_pseudoinv}
        \vy_m^{{\MR}} = - (\bHHline_m^* \bHHline_m)^\inv \bHHline_m^* \vgline_m.
    \end{equation}
    Expanding $\bHHline_m^* \bHHline_m$ results in
    \begin{align}
        \bHHline_m^* \bHHline_m
        & = (\HHline_m^* \otimes \IIline^* + \IIline^* \otimes \HHline_m^*) (\HHline_m \otimes \IIline + \IIline \otimes \HHline_m) \notag \\
        & = \HHline_m^* \HHline_m \otimes \IIline^* \IIline + \IIline^* \IIline \otimes \HHline_m^* \HHline_m + \HHline_m^* \IIline \otimes \IIline^* \HHline_m + \IIline^* \HHline_m \otimes \HHline_m^* \IIline \notag \\
        & = \HH_m^*\HH_m \otimes \II +  \II \otimes \HH_m^*\HH_m + \HH_m \otimes \HH_m^* + \HH_m^* \otimes \HH_m  + \JJ_m \otimes \II + \II \otimes \JJ_m \notag \\
        & = \bHH_m^* \bHH_m + \JJ_m \otimes \II + \II \otimes \JJ_m. \label{eq:bHHline*bHHline}
    \end{align}
    By \cite[Lemma~4.3.1]{HorJ91}, it follows that
    \begin{align}
        \bHHline_m^* \vgline_m
        & = \left( (\HHline_m^* \otimes \IIline^* ) + ( \IIline^* \otimes \HHline_m^* ) \right) \vec{\vEline_1 \Gamma \Gamma^* \vEline_1^*} \notag \\
        & = \vec{\IIline^* \vEline_1 \Gamma \Gamma^* \vEline_1^* \bar{\HHline}_m} + \vec{\HHline_m^* \vEline_1 \Gamma \Gamma^* \vEline_1^* \IIline} \notag \\
        & = \vec{\vE_1 \Gamma \Gamma^* \vE_1^* \bar{\HH}_m} + \vec{\HH_m^* \vE_1 \Gamma \Gamma^* \vE_1^*} \notag \\
        & = \left( (\HH_m^* \otimes \II ) + ( \II \otimes \HH_m^* ) \right) \vec{\vE_1 \Gamma \Gamma^* \vE_1^*} \notag \\
        & = \bHH_m^* \vg_m, \label{eq:bHHline*vgline}
    \end{align}
        where $\bar{\HH} = (\HH^*)^T$ denotes the the element-wise complex conjugate of $\HH$.  Combining \eqref{eq:bHHline*bHHline} and \eqref{eq:bHHline*vgline} with \eqref{eq:bHHline_pseudoinv} leads to the desired result:
    \begin{align*}
        \vy_m^{{\MR}}
        & = - \left( \bHH_m^* \bHH_m + \JJ_m \otimes \II + \II \otimes \JJ_m \right)^\inv \bHH_m^* \vg_m \\
        & = - \left( \bHH_m + \bHH_m^\cinv \big(\JJ_m \otimes \II + \II \otimes \JJ_m\big) \right)^\inv \vg_m \\
        & = - (\bHH_m + \bMM_m^{\MR})^\inv \vg_m.
    \end{align*}
\end{proof}

The modification matrices $\bMM_m^{\PMR}$ and $\bMM_m^{\MR}$ are indeed closely related.
\begin{lemma}\label{lem:PMR_to_MR}
    With the notation above, we have
    \begin{equation}
        \bMM_m^{\PMR} = \bMM_m^{\MR} + \bHH_m^\cinv (\HH_m^\cinv \JJ_m \otimes \HH_m^* + \HH_m^* \otimes \HH_m^\cinv \JJ_m).
    \end{equation}
\end{lemma}
\begin{proof}
	Multiplying $\bMM_m^{\PMR}$ and $\bHH_m^*$ gives
	\[
		\bHH_m^* \bMM_m^{\PMR}
		= \JJ_m \otimes \II + \II \otimes \JJ_m
		+ \HH_m^\cinv \JJ_m \otimes \HH_m^* + \HH_m^* \otimes \HH_m^\cinv \JJ_m.
	\]
	Inverting $\bHH_m^*$ immediately gives the desired result.
\end{proof}

Lemma~\ref{lem:PMR_to_MR} implies that whenever the term $\bHH_m^\cinv (\HH_m^\cinv\JJ_m \otimes \HH_m^* + \HH_m^* \otimes \HH_m^\cinv\JJ_m)$ is small, the PMR solution is close to the MR one. This is the case, for instance, when $\normF{\JJ_m} \ll 1$. However, another interesting scenario where $\bMM_m^{\PMR} \approx \bMM_m^{\MR}$ is described in the following theorem.

\begin{theorem}\label{thm:MR_PMR_distance}
     Let $\HH_m = \QQ_m \Lambda_m \QQ_m^{-1}$, $\Lambda_m = \diag(\lambda_1,\ldots,\lambda_{mr})$ and $\QQ_m\in \spC^{mr\times mr}$, be the eigendecomposition of $\HH_m$. Then 
    \begin{equation}
    	\normF{\bMM_m^{\MR}-\bMM_m^{\PMR}}^2
    	\leq  2m^3r^4 q_m \bar{c_m} \max_{i,j=1,\ldots,mr} \frac{(\abs{\lambda_j}^2 + \abs{\lambda_i}^2)^2}{\abs{\lambda_j}^2 \abs{\lambda_i}^2 |\lambda_i + \lambda_j|^2}, 
    \end{equation}
    where $\bar{c}_m> 0$, and $q_m = \normF{\QQ_m^\cinv \otimes \QQ_m^\cinv}^2$.
\end{theorem}

\begin{proof}
Again applying \cite[Lemma~4.3.1]{HorJ91}, the $p$-th column of $\bMM^{\PMR}$ (cf.~\eqref{eq:bMM_PMR}) can be rewritten as
\begin{align*}
	\bMM_m^{\PMR} \ve_p
	= & (\QQ_m^\cinv \otimes \QQ_m^\cinv)(\Lambda_m^\cinv \QQ_m^*\JJ_m \otimes \QQ_m^*+ \QQ_m^* \otimes \Lambda_m^\cinv \QQ_m^*\JJ_m )\ve_p \\
	= & (\QQ_m^\cinv \otimes \QQ_m^\cinv)(\Lambda_m^\cinv \QQ_m^*\JJ_m \otimes \QQ_m^* + \QQ_m^* \otimes \Lambda_m^\cinv \QQ_m^*\JJ_m )\vec{\ve_\ell \ve_k^T} \\
	= & (\QQ_m^\cinv \otimes \QQ_m^\cinv) \vec{\QQ_m^* \ve_\ell \ve_k^T \JJ_m^T \QQbar_m \Lambda_m^\cinv + \Lambda_m^\cinv \QQ_m^*\JJ_m \ve_\ell \ve_k^T \QQbar},
\end{align*}
where $p = (\ell-1)mr + k$, for some $\ell,k \in \{1,\ldots,mr\}$\footnote{Notice the abuse of notation here: $\ve_p\in\mathbb{R}^{m^2r^2}$, whereas $\ve_\ell,\ve_k\in\mathbb{R}^{mr}$.}, and $\QQbar_m$ denotes the element-wise complex conjugate of $\QQ_m$.

Similarly for the $p$-th column of $\bMM_m^{\MR}$ (cf.~\eqref{eq:bMM_MR}):
\begin{align*}
	\bMM_m^{\MR} \ve_p
	= & (\QQ_m^\cinv \otimes \QQ_m^\cinv) (\Lambda_m \otimes \II + \II \otimes \Lambda_m)^\cinv(\QQ_m^*\JJ_m \otimes \QQ_m^* + \QQ_m^* \otimes \QQ_m^*\JJ_m )\ve_p \\
	= & (\QQ_m^\cinv \otimes \QQ_m^\cinv) (\Lambda_m \otimes \II + \II \otimes \Lambda_m)^\cinv(\QQ_m^*\JJ_m \otimes \QQ_m^* + \QQ_m^* \otimes \QQ_m^*\JJ_m )\vec{ \ve_\ell \ve_k^T}\\
	= & (\QQ_m^\cinv \otimes \QQ_m^\cinv) \vec{\LL_m \circ (\QQ_m^* \ve_\ell \ve_k^T \JJ_m^T \QQbar_m + \QQ_m^*\JJ_m \ve_\ell \ve_k^T \QQbar_m)},
\end{align*}
where $\LL_m$ is a Cauchy matrix whose $(i,j)$-th element is given by $\LL_{i,j} = (\bar{\lambda}_i + \bar{\lambda}_j)^\inv$, and $\circ$ denotes the Hadamard product.  Note that $\LL_m$ is well defined, due to $\HH_m$ being nonsingular.

Thanks to the expression above, we first notice that the indexes of some of the zero columns of $\bMM_m^{\MR}$ and $\bMM_m^{\PMR}$ coincide. Indeed, by recalling that $\JJ_m = \vE_mH_{m+1,m}^*H_{m+1,m}\vE_m^*$, $\JJ_m \ve_\ell \ve_k^T = 0$ if both $\ell$ and $k$ belong to $\{1,\ldots,(m-1)r\}$. 

As a consequence, $\bMM_m^{\PMR} \ve_p$ and $\bMM_m^{\MR} \ve_p$ may be different from the zero vector if and only if $p = (\ell-1)mr + k$ is such that either $\ell$ or $k$ is in $\{(m-1)r+1, \ldots, mr\}$.  We thus have to consider only these $mr^2$ columns for achieving an upper bound on the column-wise differences between $\bMM_m^{\PMR}$ and $\bMM_m^{\MR}$.  Let $\DD_m := \QQ_m^*\JJ_m \ve_\ell \ve_k^T \QQbar_m$ for either $\ell$ or $k \in \{(m-1)r+1, \ldots, mr\}$.  Then
\[
    \normF{\bMM_m^{\MR} \ve_p -\bMM^{\PMR}_m \ve_p}^2
	\leq\normF{\QQ_m^\cinv\otimes \QQ_m^\cinv}^2 \normF{\LL_m \circ (\DD_m^T + \DD_m) - (\DD_m^T \Lambda^\cinv + \Lambda^\cinv \DD_m)}^2,
\]
otherwise, we are sure that $\bMM_m^{\MR} \ve_p = \bMM_m^{\PMR} \ve_p=0$.

A direct inspection of the entries of $\LL_m \circ (\DD_m^T + \DD_m)$ and $\DD_m^T \Lambda^\cinv + \Lambda^\cinv \DD_m$ shows that 
\[
\ve_i^T (\LL_m \circ (\DD_m^T + \DD_m)) \ve_j
= \frac{(\DD_m)_{j,i} + (\DD_m)_{i,j}}{\bar{\lambda}_i + \bar{\lambda}_j},
\]
and
\[
\ve_i^T (\DD_m^T \Lambda^\cinv + \Lambda^\cinv \DD_m) \ve_j
= \frac{(\DD_m)_{j,i}}{\bar{\lambda}_i} + \frac{(\DD_m)_{i,j}}{\bar{\lambda}_j}.
\]
We thus have
\begin{align*}
	\MoveEqLeft[8] \normF{\LL_m \circ (\DD_m^T + \DD_m) - (\DD_m^T \Lambda^\cinv + \Lambda^\cinv \DD_m)}^2 \\
	& \hspace{-4em} = \sum_{i=1}^{mr} \sum_{j=1}^{mr} \abs{
	    \frac{(\DD_m)_{j,i} + (\DD_m)_{i,j}}{\bar{\lambda}_i + \bar{\lambda}_j}
	    - \frac{(\DD_m)_{j,i}}{\bar{\lambda}_i}
	    - \frac{(\DD_m)_{i,j}}{\bar{\lambda}_j}
	    }^2 \\
	& \hspace{-4em} = \sum_{i=1}^{mr} \sum_{j=1}^{mr} \abs{\frac
	    {\bar{\lambda}_j\bar{\lambda}_i( (\DD_m)_{j,i} +(\DD_m)_{i,j})-(\bar{\lambda}_i + \bar{\lambda}_j)(\bar{\lambda}_j(\DD_m)_{j,i} + \bar{\lambda}_i(\DD_m)_{i,j})}
	    {\bar{\lambda}_j \bar{\lambda}_i \cdot (\bar{\lambda}_i + \bar{\lambda}_j)}
	    }^2 \\
	& \hspace{-4em} \leq m^2 r^2 c_m(\ell,k) \max_{i,j=1,\ldots,mr}
	    \frac
	    { \abs{ 2\bar{\lambda}_j\bar{\lambda}_i - (\bar{\lambda}_i + \bar{\lambda}_j)^2 }^2 }
	    { \abs{ \bar{\lambda}_j\bar{\lambda}_i \cdot (\bar{\lambda}_i + \bar{\lambda}_j) }^2 }\\
	& \hspace{-4em} \leq m^2 r^2 c_m(\ell,k) \max_{i,j=1,\ldots,mr}
	    \frac
	    { \abs{ \bar{\lambda}_i^2 + \bar{\lambda}_j^2 }^2 }
	    { \abs{ \bar{\lambda}_j\bar{\lambda}_i \cdot (\bar{\lambda}_i + \bar{\lambda}_j) }^2 },
\end{align*}
where $c_m(\ell,k) = \norm{\DD_m}_{\max}^2$. Notice that we can drop the conjugation in the expression above thanks to the presence of the absolute value. Moreover, by recalling that~$\normF{\bMM_m^{\MR} \ve_p -\bMM^{\PMR}_m \ve_p} \neq 0$ for (at most) $mr^2$ columns, we have
\begin{align*}
	\normF{\bMM^{\MR}-\bMM^{\PMR}}^2
	= & \sum_{p=1}^{m^2r^2} \normF{\bMM^{\MR} \ve_p-\bMM^{\PMR} \ve_p}^2 \\
	\leq & 2m^3r^4 q_m \bar{c}_m \max_{i,j=1,\ldots,mr}
	    \frac
	    { \abs{ \lambda_i^2 + \lambda_j^2 }^2 }
	    { \abs{ \lambda_j\lambda_i \cdot (\lambda_i + \lambda_j) }^2 },
\end{align*}
where $\bar{c}_m = \max_{\ell,k} c_m(\ell,k)$, and $q_m = \normF{\QQ_m^\cinv \otimes \QQ_m^\cinv}$. 
\end{proof}

Theorem~\ref{thm:MR_PMR_distance}
shows that the distance between $\bMM_m^{\PMR}$ and $\bMM_m^{\MR}$ can be related to the magnitude of the function
\begin{equation} \label{eq:dist_func}
	f(x,y) := \frac{\abs{x^2+y^2}^2}{\abs{x y \cdot (x+y)}^2}, \quad x,y \in \spec(\HH_m).
\end{equation}
Since $\spec(\HH_m) \subseteq W(A)$, where $W(A):=\{z^*Az\in\mathbb{C},\;\|z\|=1\}$ is the field of values of $A$, one can compute an approximation of $W(A)$ and study $\max_{x,y} f(x,y)$ in $W(A)$. If this value is small, then we can expect the PMR and MR methods to achieve similar results.

\begin{corollary}\label{cor:MR_PMR_distance}
Let $A=A^*$. With the notation of Theorem~\ref{thm:MR_PMR_distance}, it holds
    \begin{equation}
    	\normF{\bMM_m^{\MR} - \bMM_m^{\PMR}}^2
    	\leq  2m^3r^4 \bar{c_m} \max_{i,j=1,\ldots,mr} \frac{(\lambda_j^2 + \lambda_i^2)^2}{\lambda_j^2 \lambda_i^2 (\lambda_i + \lambda_j)^2}. 
    \end{equation}
\end{corollary}

\begin{proof}
     The proof follows from Theorem~\ref{thm:MR_PMR_distance} by noticing that having an Hermitian $A$ implies $\HH_m$ to be Hermitian so that $\QQ_m$ is unitary and $\spec(\HH_m)$ is real. 
\end{proof}

\subsection{A residual-minimizing Kronecker sum approach} \label{sec:kron_sum}
We introduce $\oplus$ to denote the \emph{Kronecker sum} between two matrices $B,C \in \spC^{p \times p}$:
\begin{equation} \label{eq:kronecker_sum_def}
    B \oplus C := B \otimes I_p + I_p \otimes C.
\end{equation}

Define $\bJJ_m := \JJ_m \oplus \JJ_m $. It follows from Proposition~\ref{prop:MR_LRM_form} and the definitions in \eqref{eq:bMM_PMR} that
\[
    \vy_m^{\MR} = -\left( \bHH_m + \bHH_m^\cinv \bJJ_m \right)^\inv \vg_m.
\]
and
\[
    \vy_m^{\PMR} = -\big( (\HH_m + \HH_m^\cinv \JJ_m) \oplus (\HH_m + \HH_m^\cinv \JJ_m) \big)^\inv \vg_m.
\]

It is tempting to conjecture that $(\HH_m + \HH_m^\cinv \JJ_m) \oplus (\HH_m + \HH_m^\cinv \JJ_m)$ is the closest Kronecker sum to $\bHH_m + \bHH_m^\cinv \bJJ_m$, under the Frobenius norm.  In particular, it is possible to find $\vM_{\NKS} \in \spC^{mr \times r}$, where NKS stands for ``nearest Kronecker sum", such that
\[
    \vy_m^{\NKS} = -\big( (\HH_m + \vM^{\NKS} \vE_m^*) \oplus (\HH_m + \vM^{\NKS} \vE_m^*) \big)^\inv \vg_m,
\]
and $\vM_{\NKS}$ minimizes the Frobenius norm of \eqref{eq:res_comp_mod}, thereby potentially achieving a much lower residual for the problem \eqref{eq:lyap} than the PMR approach.  More explicitly,
\begin{equation} \label{eq:opt_nks}
    \vM_{\NKS} := \argmin \normF{G_m(\vM)
        \begin{bmatrix}
                & I_r 	&		\\
            I_r & 		& -I_r	\\
                & -I_r 	&
        \end{bmatrix}
        G_m(\vM)^*}
\end{equation}
where $G_m(\vM) :=$
\[
    \begin{bmatrix}
        \vE_{m+1} H_{m+1,m}	 -\IIline \vec{\big( (\HH_m + \vM \vE_m^*) \oplus (\HH_m + \vM \vE_m^*) \big)^\inv \vg_m} \vE_m	& \IIline \vM
    \end{bmatrix},
\]
and
$\IIline := \begin{bmatrix} I_{mr} \\ 0 \end{bmatrix}$.

Unfortunately, this conjecture is not true, as demonstrated by the script\\ \texttt{test\_kron\_sum\_conj.m} in our toolbox \texttt{LowRank4Lyap}, described in further detail in Section~\ref{sec:num_ex}.  In this script, $A$ is a one-dimensional discretization of the Laplace operator, and the optimization problem \eqref{eq:opt_nks} is solved via a general unconstrained optimization routine \texttt{fminunc} with $\vM^{\PMR}$ as an initial guess.  See Figure~\ref{fig:test_kron_sum_conj} for a plot of the relative residuals alongside the relative difference $\frac{\norm{\vM^{\PMR} - \vM^{\NKS}}_2}{\norm{\vM^{\PMR}}_2}$ and Table~\ref{tab:test_kron_sum_conj} for the final iterations of the plot.  Although NKS, PMR, and MR appear to overlap, there are small visible differences when one zooms in.

\begin{figure}[t!]
    \begin{center}
        \small
        \begin{tabular}{cc}
             Full trajectory & Up to $i=49$ \\
             \includegraphics[width=.45\textwidth]{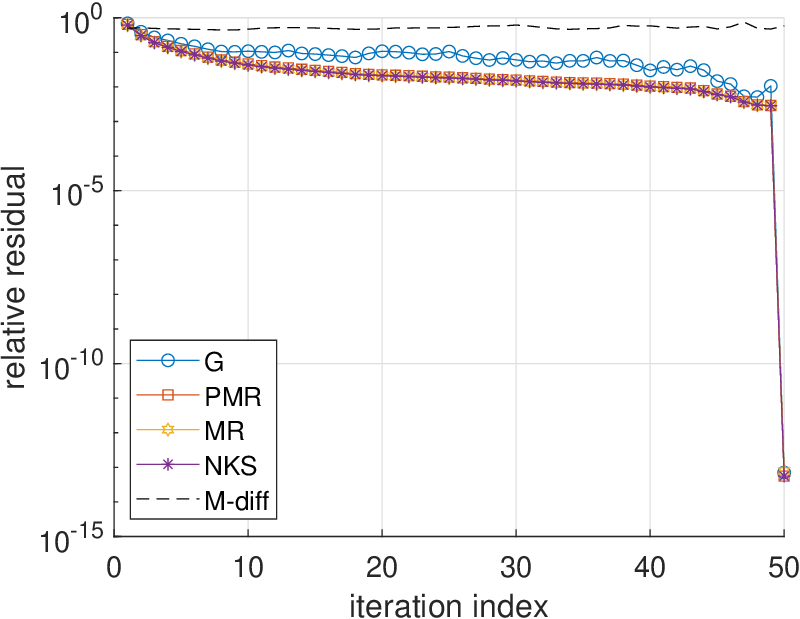} &
             \includegraphics[width=.45\textwidth]{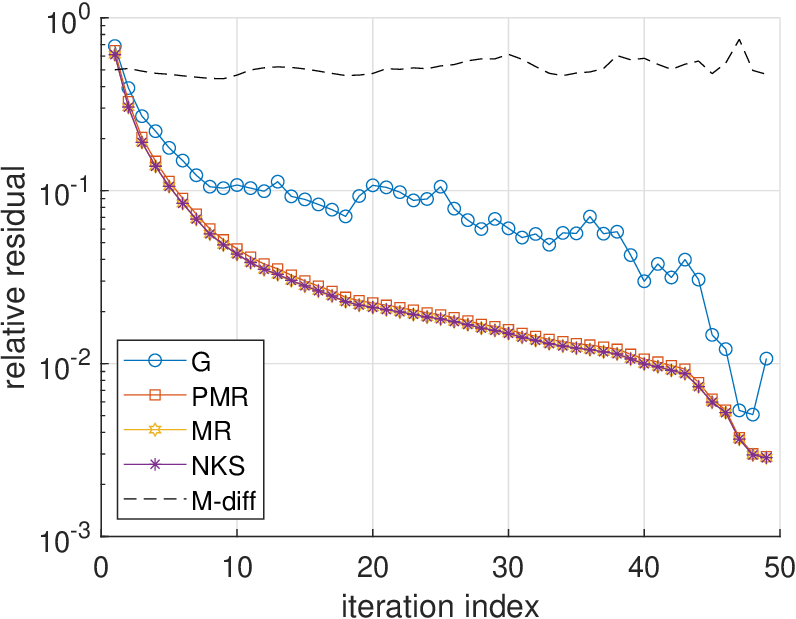}
        \end{tabular}
    \end{center}
    \caption{Relative residuals with respect to iteration index for various methods and $A$ as a 1D Laplacian.  ``M-diff" is $\frac{\norm{\vM^{\PMR} - \vM^{\NKS}}_2}{\norm{\vM^{\PMR}}_2}$. \label{fig:test_kron_sum_conj}}
\end{figure}

\begin{table}[tbp!]
    \centering
    \begin{tabular}{c|cccc}
           & Galerkin   & PMR        & MR         & NKS        \\\hline
        46 & 1.2120e-02 & 5.3767e-03 & 5.1983e-03 & 5.1996e-03 \\
        47 & 5.3607e-03 & 3.7351e-03 & 3.6560e-03 & 3.6565e-03 \\
        48 & 5.0699e-03 & 3.0104e-03 & 2.9629e-03 & 2.9632e-03 \\
        49 & 1.0686e-02 & 2.8980e-03 & 2.8543e-03 & 2.8548e-03 \\
        50 & 7.0699e-14 & 5.5865e-14 & 7.8170e-14 & 5.5956e-14 \\
    \end{tabular}
    \caption{Relative residuals for final iterations in Figure~\ref{fig:test_kron_sum_conj}. \label{tab:test_kron_sum_conj} }
\end{table}

Although the NKS approach fits the LRM framework, as it is currently implemented, it is not scalable in practice.  Furthermore, for Hermitian $A$, NKS and PMR appear to be very close, meaning that PMR may be sufficient for such scenarios.  We also examine NKS for a non-symmetric problem; see Example~\ref{ex:iss_n270_r3} in Section~\ref{sec:num_ex}.
\section{Compress-and-restart strategy} \label{sec:cr}
Following earlier work with Kressner and Massei~\cite{KreLMetal21}, we note that a compress-and-restart strategy is also viable for the LRM framework in Section~\ref{sec:lrm_framework}.  The key is to express the residual in a factored form with a ``core" matrix.  For the following, we drop the $\mbox{\normalfont\sffamily mod}$ superscript to allow space for restart cycle indices, expressed as $\k$, $k = 1, \ldots, k_{\max}$.  The cycle length and block size are denoted as $m_k$ and $r_k$, respectively, and in the case of redundancies, they may be dropped in favor of just the cycle index.

From Proposition~\ref{prop:res_lrm_galerkin}, recall that 
\begin{equation}
    R^\one := \bVV_{m_1+1} G^\one L^\one \big(G^\one)^* \bVV_{m_1+1}^* \in \spC^{n \times n},
\end{equation}
where for all $k \in \{1, \ldots, k_{\max}\}$
\begin{equation}
    L^\k := 
    \begin{bmatrix}
            & I_{r_k}   &   		\\
    I_{r_k} &           & -I_{r_k}	\\
	    	& -I_{r_k} 	&
	\end{bmatrix}.
\end{equation}
Denote $\vC^\two := \bVV_{m_1+1} G^\one \in \spC^{n \times r_2}$, where here $r_2 = 3r_1$.  To restart, we compute the next Krylov subspace $\spK_{m_2}(A, \vC^\two)$ and obtain the basis $\bVV_{m_2+1} \in \spC^{n \times (m_2+1) r_2}$ and block Hessenberg $\HHline_{m_2+1} \in \spC^{(m_2+1)r_2 \times m_2 r_2}$.  To compute an update to the solution from $k = 1$, we then solve the projected problem
\begin{equation} \label{eq:restart_k=2}
    \big(\HH_{m_2} + \vM^\two \vE_{m_2}^* \big) Y + Y \big(\HH_{m_2} + \vM^\two \vE_{m_2}^* \big)^* + \vE_1 \Gamma^\two L^\two \big( \Gamma^\two \big)^* \vE_1^* = 0,
\end{equation}
where $\vM^\two \in \spC^{m_2 r_2 \times r_2}$ is an LRM and here $\vE_1 \in \spR^{n \times m_2 r_2}$.  Letting $Y^\two$ denote solution to \eqref{eq:restart_k=2}, we can then compute $Z^\two := \bVV_{m_2} Y^\one \bVV_{m_2}^*$ and add it back to $X^\one$.  Doing this iteratively leads to a final solution of the form
\[
X = X^\one + \sum_{k=2}^{k_{\max}} Z^\k,
\]
where each $Z^\k = \bVV_{m_k} Y^\k \bVV_{m_k}^*$, $Y^\k$ is the solution of the projected problem
\begin{equation}
    \big(\HH_{m_k} + \vM^\k \vE_{m_k}^* \big) Y + Y \big(\HH_{m_k} + \vM^\k \vE_{m_k}^* \big)^* + \vE_1 \Gamma^\k L^\k \big( \Gamma^\k \big)^* \vE_1^* = 0,
\end{equation}
$\bVV_{m_k}^* \vC^\k = \vE_1 \Gamma^\k$,
\[
    \vC^\k := \bVV_{m_\kless+1} G^\kless \in \spC^{n \times r_\kless},
\]
and
\[
    G^\k :=
    \begin{bmatrix}
        \vE_{m_k+1} H_{m_k+1,m_k}	& \IIline Y^\k \vE_{m_k}	& \IIline \vM^\k
    \end{bmatrix}
    \in \spC^{n \times 3r_k}.
\]
Although the residual at each iteration can be computed as in Proposition~\ref{prop:res_lrm_galerkin},
\begin{equation}
    R^\k := \bVV_{m_k+1} G^\k L^\k \big(G^\k)^* \bVV_{m_k+1}^* \in \spC^{n \times n},
\end{equation}
one clear downside to this approach is that the rank of $R^\k$ triples each restart cycle.  As discussed in more detail in \cite{KreLMetal21}, we therefore need a compression strategy of the starting vector $\vC^\k$ from one cycle to the next.  We do this in a symmetric fashion via \cite[Algorithm~3]{KreLMetal21}.  We also employ a $\memmax$ parameter that dictates how many column vectors can be stored per cycle and toggle the block size $r_k$ and maximum basis size $m_k$ according per cycle.  See \cite[Algorithm~4]{KreLMetal21} for more details.
\section{Numerical results}\label{Numerical results} \label{sec:num_ex}
All numerical tests were written and run in MATLAB 2022a and can be found in the repository \texttt{LowRankMod4Lyap}\footnote{\url{https://gitlab.com/katlund/LowRankMod4Lyap}} hosted on GitLab.  Every test was run on a single, standard node of the compute cluster Mechthild\footnote{\url{https://www.mpi-magdeburg.mpg.de/cluster/mechthild}}, housed at the Max Planck Institute for Dynamics of Complex Technical Systems in Magdeburg, Germany.  A standard node consists of two Intel Xeon Silver 4110 (Skylake) CPUs, each with 8 cores, 64KB L1 cache, and 1024KB L2 cache at a clockrate of 2.1 GHz, as well as 12 MB of shared L3 cache.  We set \texttt{maxNumCompThread} in MATLAB to $4$.

We consider a wide variety of numerical tests in this section.  In Section~\ref{sec:conv_behavior} we compare the convergence results among the G, MR, and PMR methods for both Hermitian and non-Hermitian matrices $A$ and plot the bound from Theorem~\ref{thm:MR_PMR_distance} for problems with Hermitian $A$.  We study the performance of these methods for varying ranks $r$ in Section~\ref{sec:vary_r}.  We see how the compress-and-restart strategy works for PMR compared to G in Section~\ref{sec:cr_num}.

Test matrices have either been generated by our own code or taken from the SuiteSparse Matrix Collection \cite{DavH11} (and in particular, originally from the Oberwolfach Benchmark Collection \cite{KorR05}) or the SLICOT Benchmark Collection \cite{ChaV05}.\footnote{The collection can currently be found at \url{https://github.com/SLICOT/Benchmark-ModelReduction}.}  We provide descriptions for each matrix below:
\begin{itemize}
    \item \badconddiag: $A$ is diagonal matrix with logarithmically spaced values ranging from $1$ to $10^{12}$.
    \item \convdiff: central finite differences stencil of three-dimensional convection-diffusion operator.  The matrix $A$ has size $n=N^3$, where $N$ is the number of discretization points in each direction.  The parameter $\eps$ controls the viscosity-dominance, whereby smaller $\eps$ correlates with low viscosity and high convection.
    \item \laplacian: central finite differences stencil of the two-dimensional Laplacian operator.  The matrix $A$ has size $n=N^2$, where $N$ is the number of discretization points in each direction.
    \item \logdiag: Example 4 from \cite{PalS20}; $A$ is nonnormal but diagonalizable with logarithmically spaced eigenvalues from $1$ to $100$.
    \item \rail: a symmetric, heat transfer, steel profile cooling matrix of size $1357 \times 1357$; $\vC$ is determined by the problem; in the Oberwolfach collection.
    \item \iss: component 1r of the international space station problem, nonsymmetric; $\vC$ is determined by the problem; in the SLICOT collection.
\end{itemize}
Unless otherwise noted, the constant term $\vC \vC^*$ is built from uniformly distributed random numbers (i.e., \texttt{rand} in MATLAB).

\subsection{Convergence behavior} \label{sec:conv_behavior}
For all examples in this section, we restrict ourselves to the non-restarted version of each algorithm.  In addition to the convergence results, we also plot the eigenvalues of the solution $Y_m$ to the small problem \eqref{eq:lyap_projected}, \eqref{eq:lyap_projected_MR}, or \eqref{eq:lyap_mod_galerkin}, as well as a Ritz-value function based on the bound from Theorem~\ref{thm:MR_PMR_distance},
\begin{equation} \label{eq:ritz_val_func}
    f(x,y) := \frac{\abs{x^2+y^2}^2}{\abs{x y \cdot (x+y)}^2},
\end{equation}
where the coefficients have been dropped.  In each example, the function is evaluated on the spectrum of $\HH_m$.  All algorithms are halted after surpassing a relative residual of $10^{-6}$.

\begin{example} \label{ex:log_diag_rand_n1000_r3}
    For the first example, we consider a \logdiag problem with $n = 1000$ and $r = 3$. Notice that, in this example, $A$ has field of values in the right half of the complex plane $\spC^{+}$, so the exact solution $X$ to~\eqref{eq:lyap} is negative semi-definite. 
    
    The results are shown in Figure~\ref{fig:log_diag_rand} and all methods perform similarly. However, MR suffers from a slightly positive solution eigenvalue.  Neither the Galerkin nor PMR approaches produce positive eigenvalues as expected from our theoretical analysis.  The Ritz-valued function shows that the difference between PMR and MR should be relatively small, confirming what we see for the residual behavior.
\end{example}
\begin{figure}[htbp!]
    \begin{center}
        \small
        \begin{tabular}{cc}
            Residual & Solution spectra \\
            \includegraphics[width=.45\textwidth]{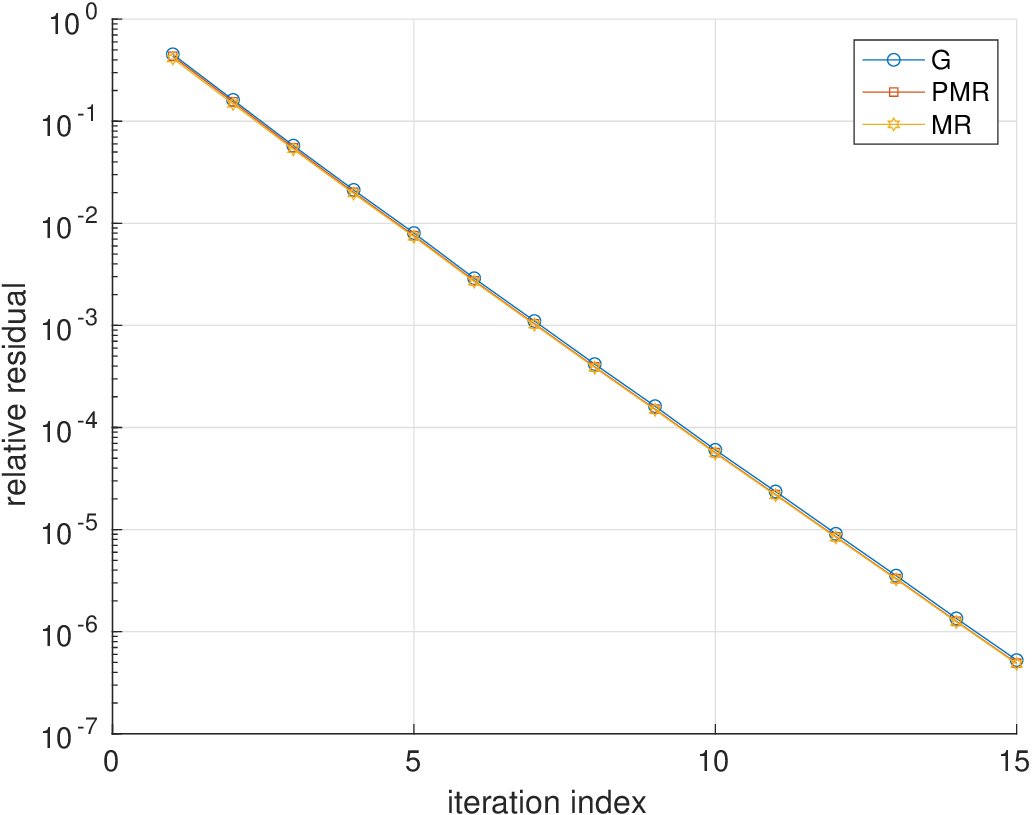} &
            \includegraphics[width=.45\textwidth]{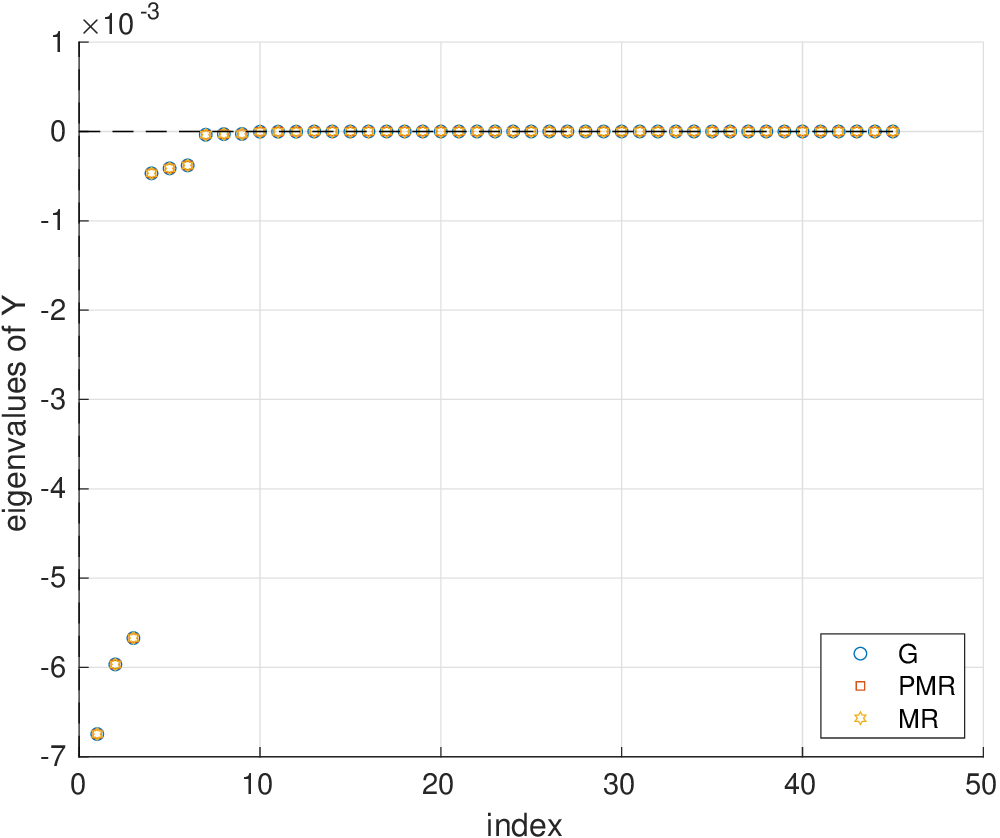}
        \end{tabular}
        \begin{tabular}{c}
            Ritz-value function \\
            \includegraphics[width=.45\textwidth]{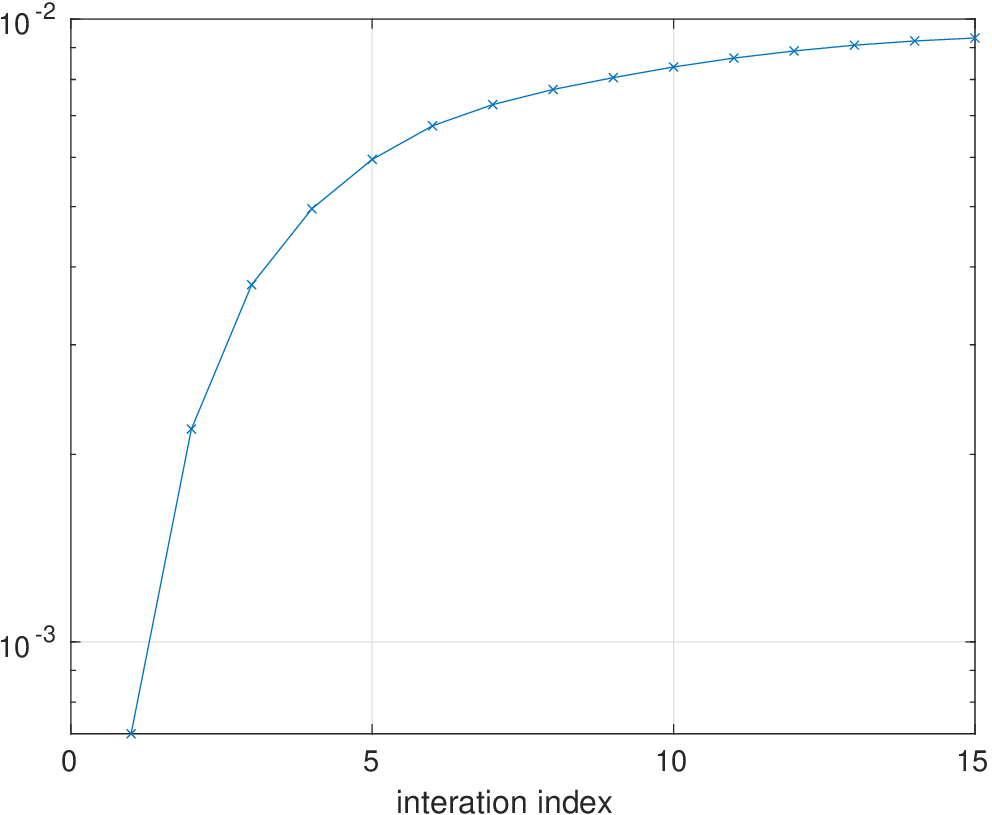}
        \end{tabular}
    \end{center}
    \caption{Convergence results for Example~\ref{ex:log_diag_rand_n1000_r3}. \label{fig:log_diag_rand}}
\end{figure}

\begin{example} \label{ex:bad_cond_diag_rand_n500_r3}
    We now consider a \badconddiag problem with $n = 500$ and $r = 3$, with results shown in Figure~\ref{fig:bad_cond_diag_rand}.  Both MR and PMR achieve nearly the same relative residual in the early iteration and improve over the Galerkin approach until later residuals, at which point all methods overlap.  The cluster of eigenvalues of $A$ near zero poses numerical challenges for all algorithms, leading to some positive eigenvalues in $Y_m$.  However, the PMR approach minimizes the positive eigenvalues the best with the largest having magnitude $\bigO{10^{-16}}$, while that of the Galerkin approach is $\bigO{10^{-11}}$ and that of MR is $\bigO{10^{-7}}$.  As for the Ritz-value function, it remains relatively small until near convergence, at which point $\JJ_m$ is close to zero, causing a spike.
\end{example}
\begin{figure}[htbp!]
    \begin{center}
        \small
        \begin{tabular}{cc}
            Residual & Solution spectra \\
            \includegraphics[width=.45\textwidth]{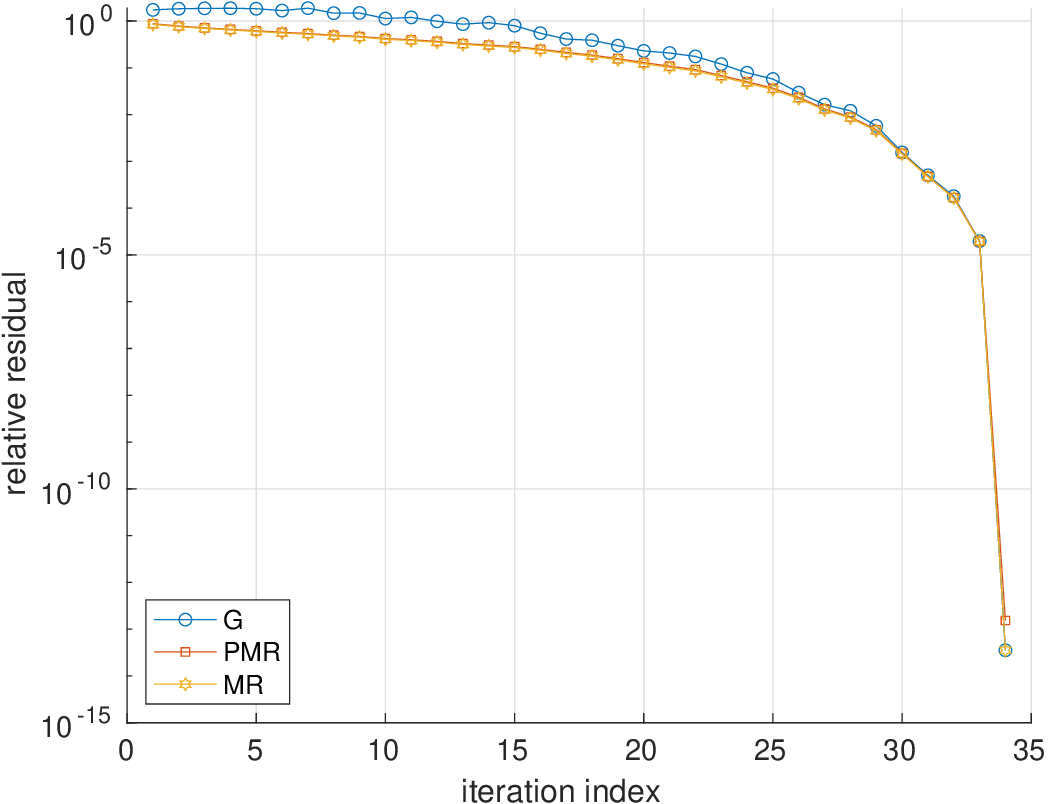} &
            \includegraphics[width=.45\textwidth]{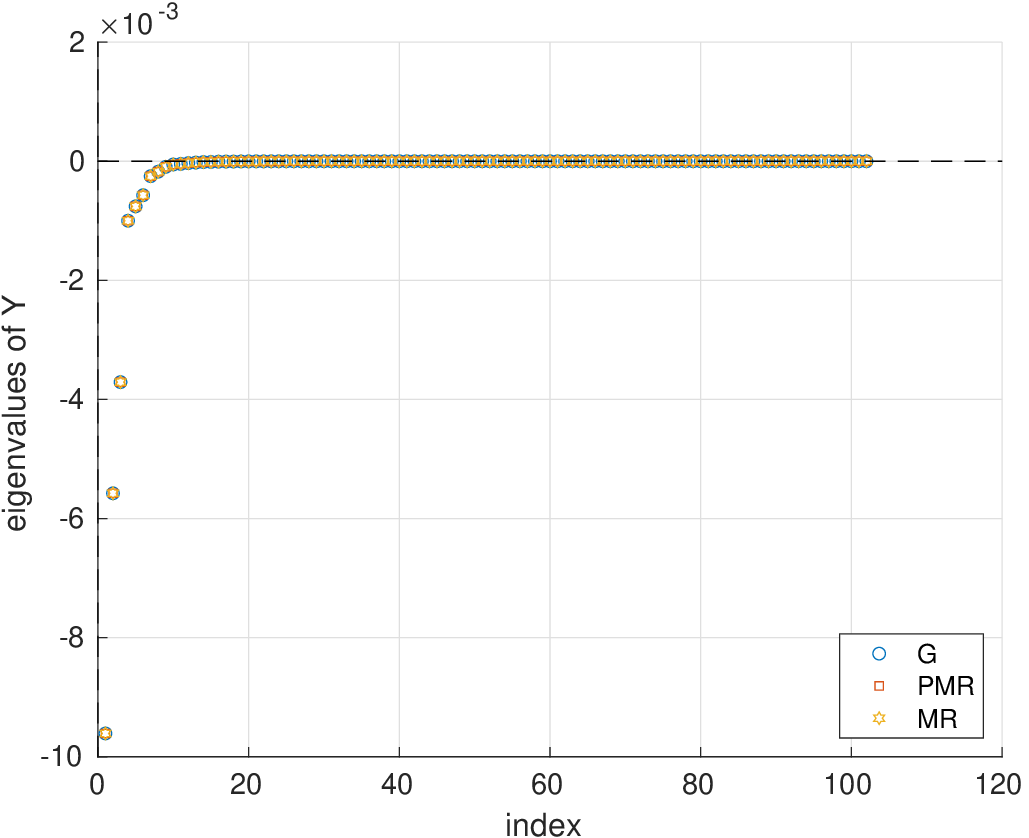}
        \end{tabular}
        \begin{tabular}{c}
            Ritz-value function \\
            \includegraphics[width=.45\textwidth]{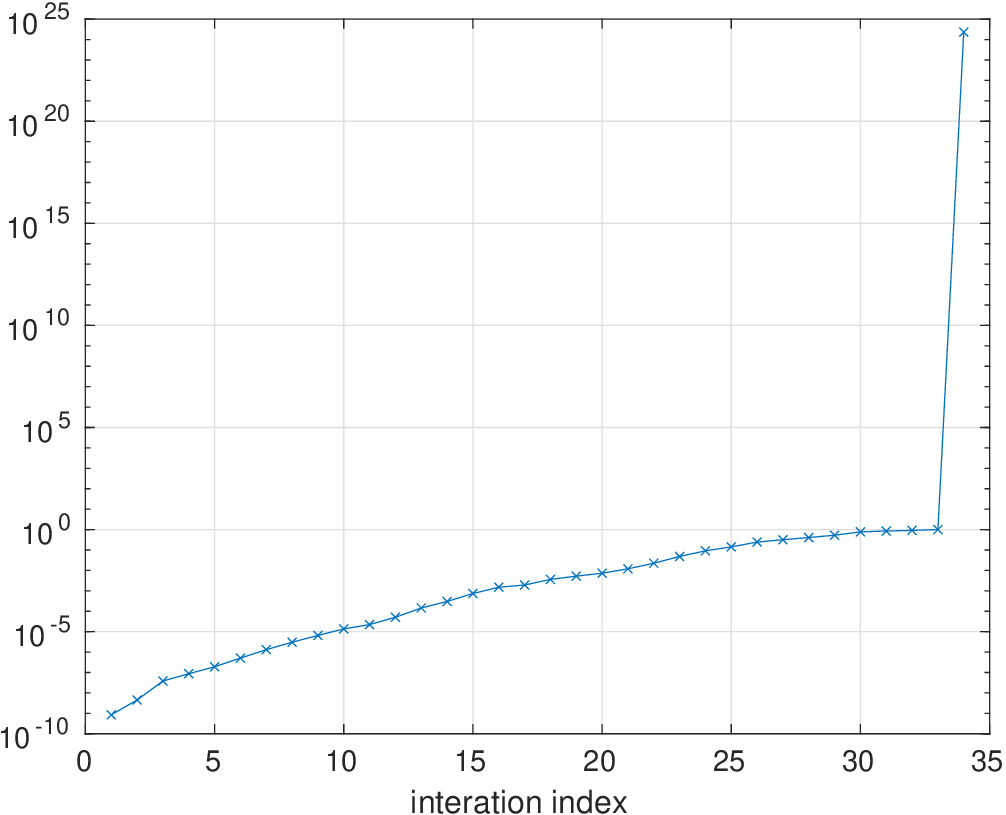}
        \end{tabular}
    \end{center}
    \caption{Convergence results for Example~\ref{ex:bad_cond_diag_rand_n500_r3}. \label{fig:bad_cond_diag_rand}}
\end{figure}

\begin{example} \label{ex:laplacian_2d_rand_n100_r3}
    We begin to see a clearer advantage of the PMR approach for the \laplacian problem with $n = 10^4$ and $r = 3$ in Figure~\ref{fig:laplacian_2d_rand}.  The PMR method overlaps with MR very closely and reaches $10^{-6}$ 10 iterations before the Galerkin method.  As for the solution spectra, there are fewer eigenvalues for MR and PMR due to the earlier convergence, but we can see that all methods closely approximate the same spectra overall.  The Ritz-value function also remains relatively small for all iterations.
\end{example}
\begin{figure}[htbp!]
    \begin{center}
        \small
        \begin{tabular}{cc}
            Residual & Solution spectra \\
            \includegraphics[width=.45\textwidth]{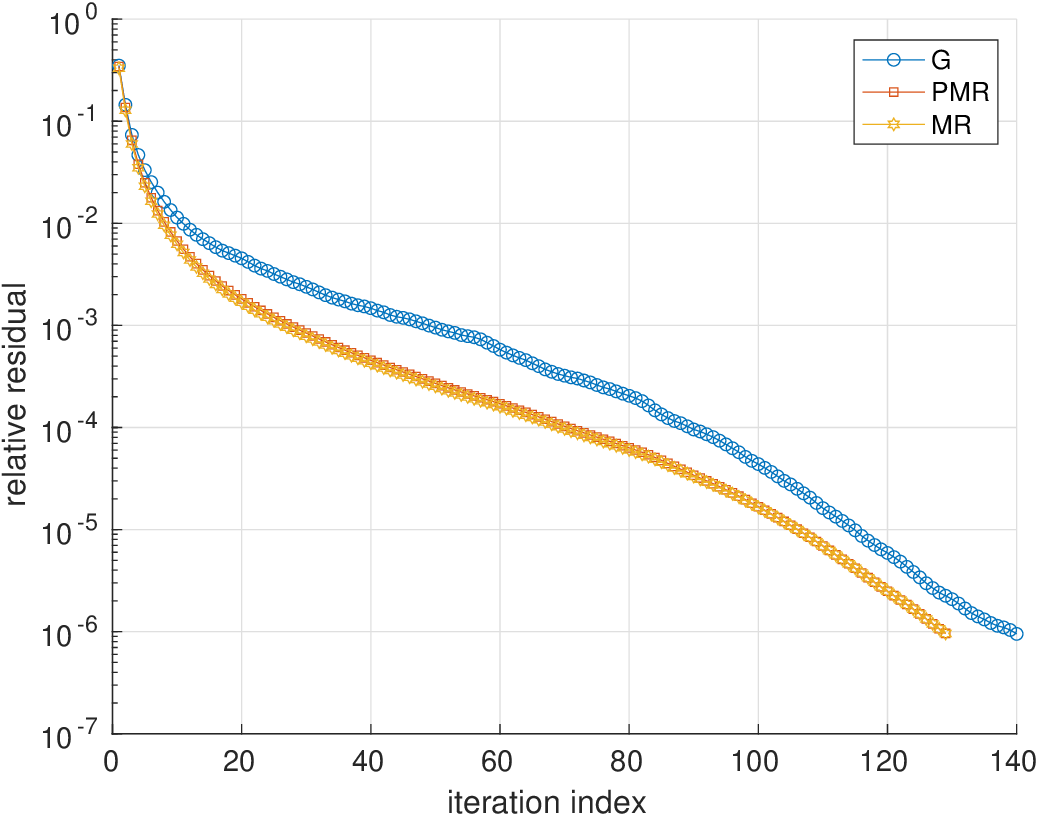} &
            \includegraphics[width=.45\textwidth]{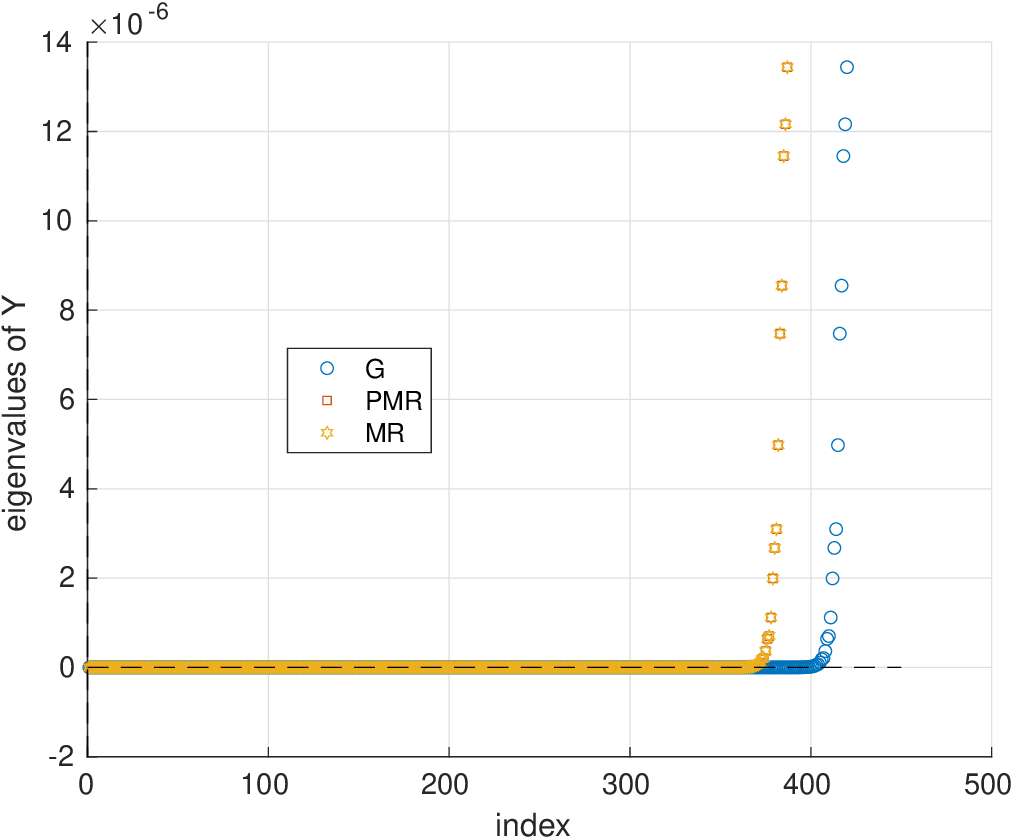}
        \end{tabular}
        \begin{tabular}{c}
            Ritz-value function \\
            \includegraphics[width=.45\textwidth]{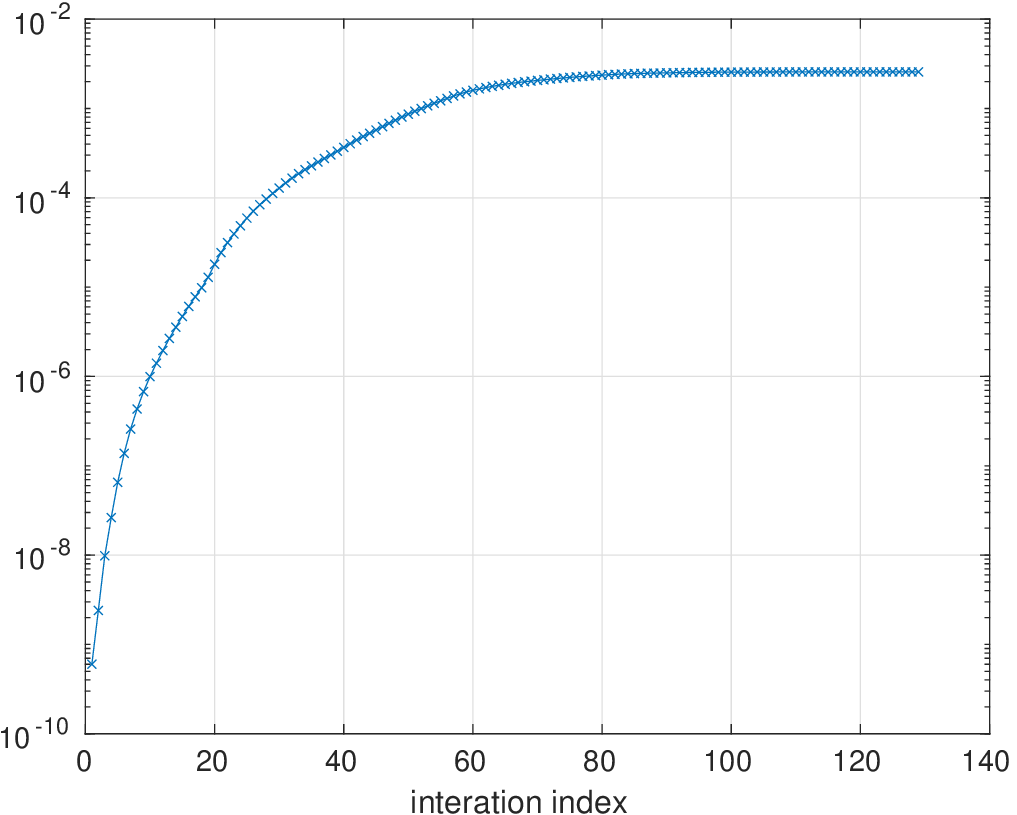}
        \end{tabular}
    \end{center}
    \caption{Convergence results for Example~\ref{ex:laplacian_2d_rand_n100_r3}. \label{fig:laplacian_2d_rand}}
\end{figure}

\begin{example} \label{ex:conv_diff_3d_rand_n25_r3}
    The \convdiff problem, with $n = 25^3$ and $r = 3$ allows us to see the effect nonnormality has on the behavior of each method.  For a problem with strong diffusion ($\eps = 10^{-2}$, left in Figure~\ref{fig:conv_diff_3d_rand}), we see that PMR overlaps with MR relatively well and converges one iteration before the Galerkin method.  As diffusion becomes weaker ($\eps = 10^{-3}$, right in Figure~\ref{fig:conv_diff_3d_rand}), PMR drifts from MR and is unable to capture its more drastic reduction in iterations (9 fewer vs. 3 fewer, for MR and PMR, respectively).  This confirms our intuition that PMR better approximates MR for symmetric matrices $A$.
\end{example}
\begin{figure}[htbp!]
    \begin{center}
        \small
        \begin{tabular}{cc}
            Residual, $\epsilon = 10^{-2}$ & Residual, $\epsilon = 10^{-3}$ \\
            \includegraphics[width=.45\textwidth]{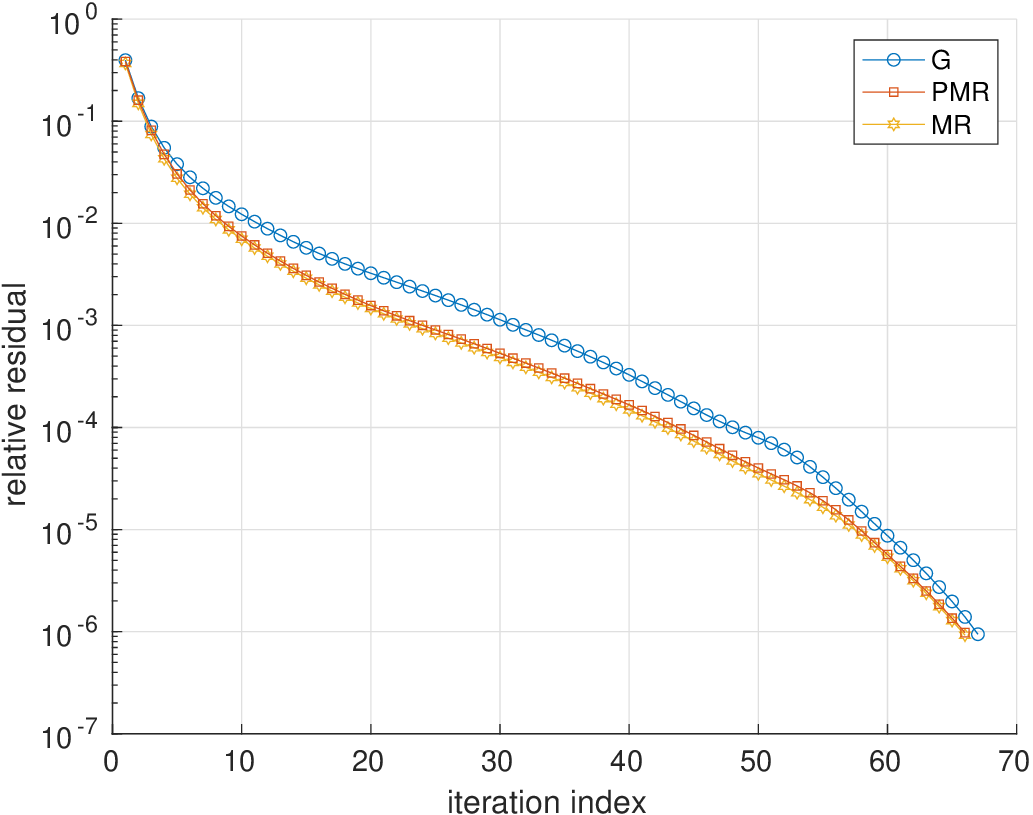} &
            \includegraphics[width=.45\textwidth]{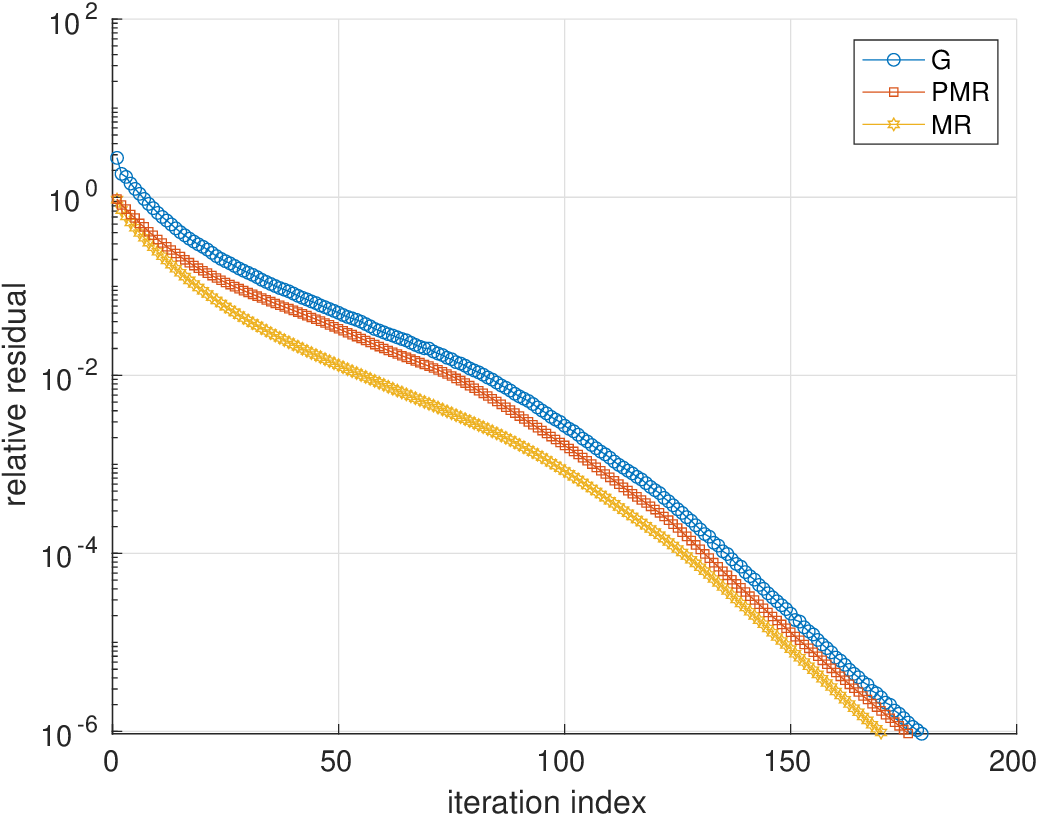}
        \end{tabular}
    \end{center}
    \caption{Convergence results for Example~\ref{ex:conv_diff_3d_rand_n25_r3}. \label{fig:conv_diff_3d_rand}}
\end{figure}

\begin{example} \label{ex:rail_1357_n1357_r7}
    With the \rail problem, we get a look at a more practical application.  Although not reported in the plots in Figure~\ref{fig:rail_1357}, we first note that the Galerkin and PMR methods both took about 2 minutes, while MR needed more than 35 minutes, to achieve the same relative residual tolerance.  We can also clearly see that the Galerkin method does not produce a monotonic residual, while PMR follows the minimal residual of MR very closely.  In terms of iteration counts, there is practically no improvement over the Galerkin method, but residual monotonicity for the same computational cost, is incredibly useful in practice.  The Ritz-value function is, however, counter-intuitive, given how well PMR visually overlaps with MR.  This suggests that a more precise metric might be needed to better measure when PMR could be relied on to approximate MR.  No method produces erroneous solution eigenvalues.
\end{example}
\begin{figure}[htbp!]
    \begin{center}
        \small
        \begin{tabular}{cc}
            Residual & Solution spectra \\
            \includegraphics[width=.45\textwidth]{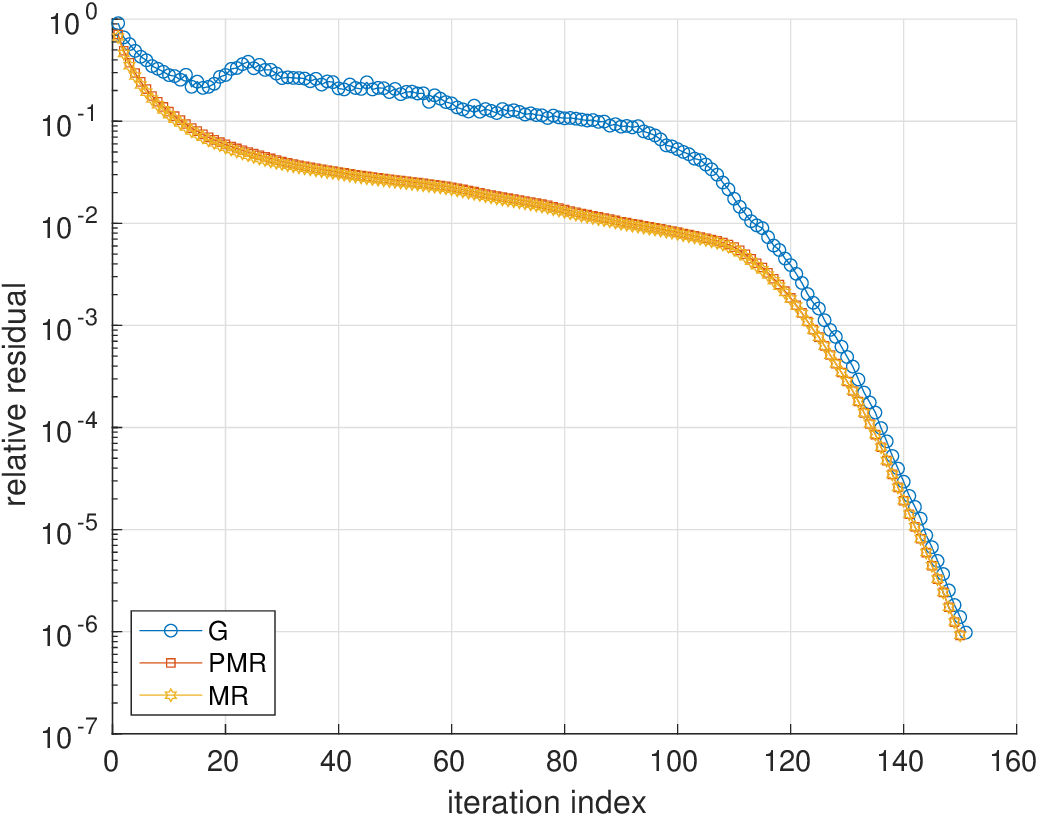} &
            \includegraphics[width=.45\textwidth]{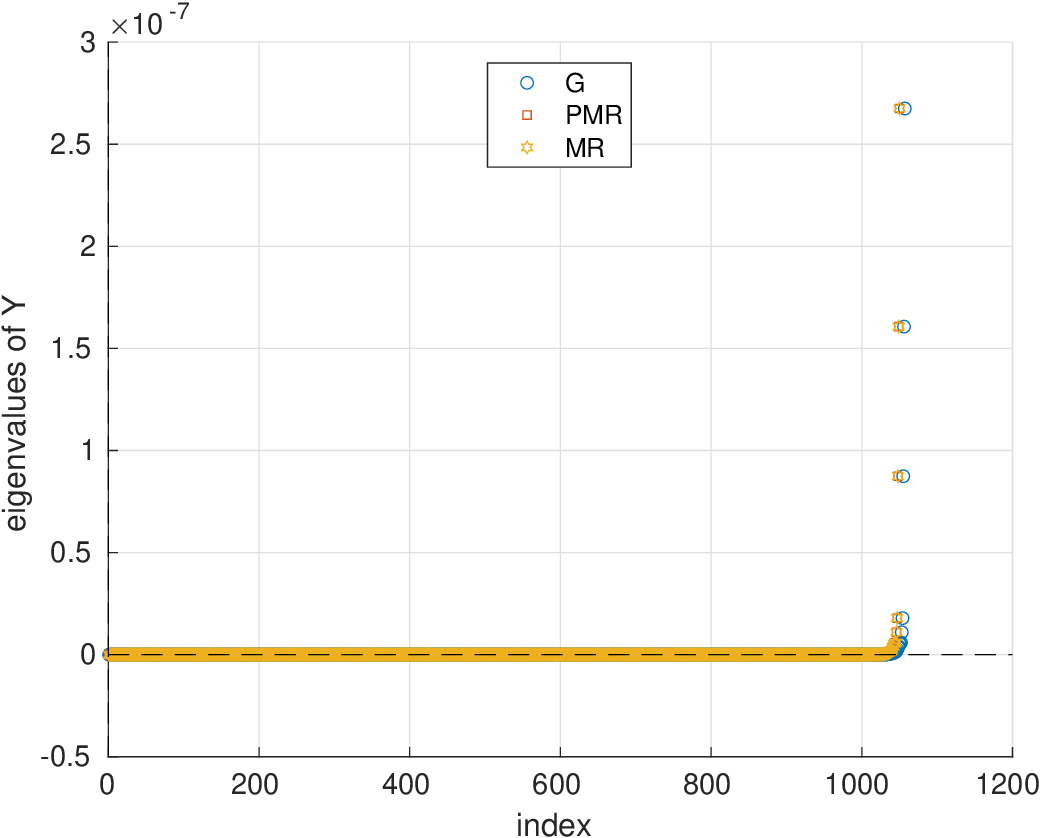}
        \end{tabular}
        \begin{tabular}{c}
            Ritz-value function \\
            \includegraphics[width=.45\textwidth]{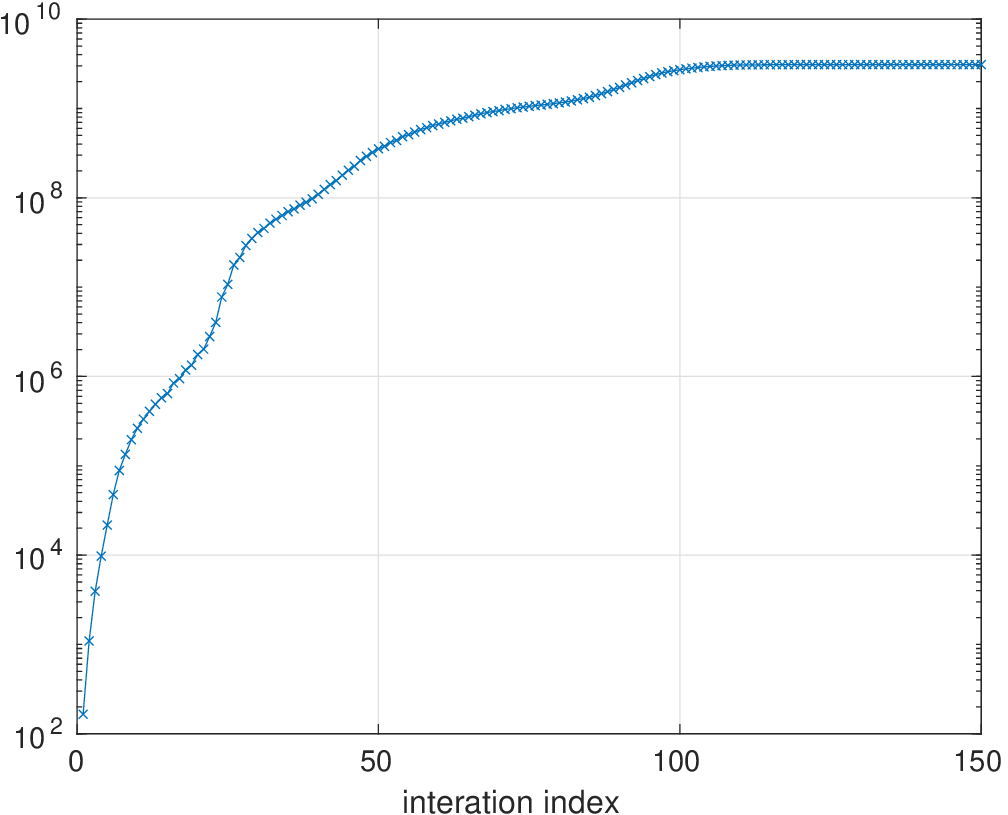}
        \end{tabular}
    \end{center}
    \caption{Convergence results for Example~\ref{ex:rail_1357_n1357_r7}. \label{fig:rail_1357}}
\end{figure}

\begin{example} \label{ex:iss_n270_r3}
    For the final example in this section, we demonstrate how badly PMR can behave when $A$ comes from the \iss problem and is far from symmetric.  Although PMR is in some ways less bad than Galerkin, both are far from monotonic, especially with respect to the MR approach.  We also compare both methods with NKS as described in Section~\ref{sec:kron_sum}.  As implemented, NKS is far from practical, because an optimization problem has to be solved each iteration.  However, it is clear that it achieves a much lower residual with fewer large jumps.
\end{example}
\begin{figure}[htbp!]
    \begin{center}
        \small
        \begin{tabular}{c}
            Residual \\
            \includegraphics[width=.65\textwidth]{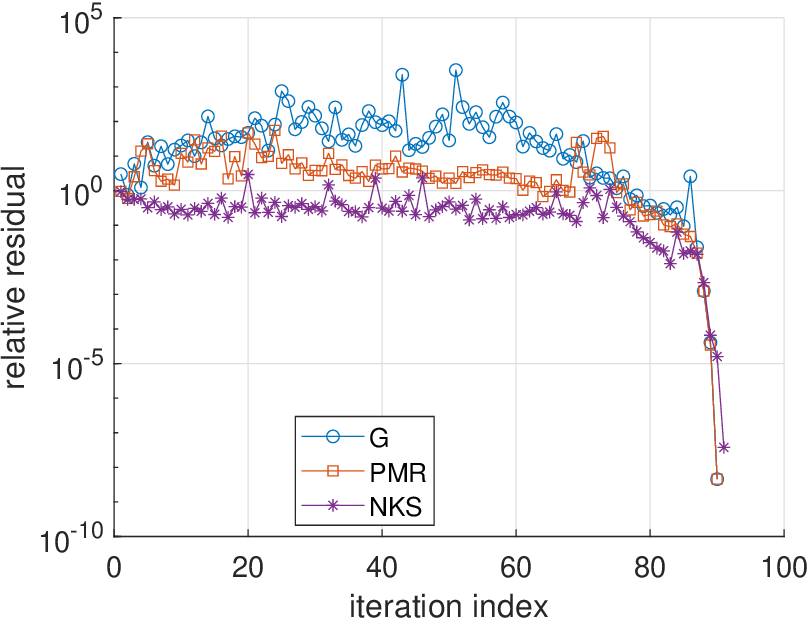}
        \end{tabular}
    \end{center}
    \caption{Convergence results for Example~\ref{ex:iss_n270_r3}. \label{fig:iss}}
\end{figure}

\subsection{Performance relative to block size} \label{sec:vary_r}
We now examine how effective each approach is relative to varying rank sizes.  We consider just the \laplacian and \convdiff problems.  All results are presented as heatmaps scaled by the results from the Galerkin approach, to facilitate comparisons.  We consider $r \in \{1, 2, 4, 8, 16\}$.  The blue-green heatmaps display timings ratios, while the orange-pink ones display a ratio of iteration counts.  All tests are run on a standard node of Mechthild as described at the beginning of Section~\ref{sec:num_ex}.

\begin{example} \label{ex:vary_r_laplacian_2d_rand_n100}
    Results for the \laplacian problem can be found in Figure~\ref{fig:vary_r_laplacian_2d_rand_n100}.  For small $r$, both MR and PMR appear to have a strong advantage over the Galerkin approach in terms of both timings and iteration counts.  In fact, MR and PMR always have fewer iteration counts.  However, MR clearly begins to suffer for $r = 4$, and its overall timings become increasingly worse as $r$ increases.  PMR in contrast remains competitive, always requiring fewer iterations than Galerkin and being slightly faster in terms of timings, although the gains decrease as $r$ increases.
\end{example}
\begin{figure}[htbp!]
    \begin{center}
        \small
        \begin{tabular}{cc}
            Timings ratio & Iteration counts ratio \\
            \includegraphics[width=.45\textwidth]{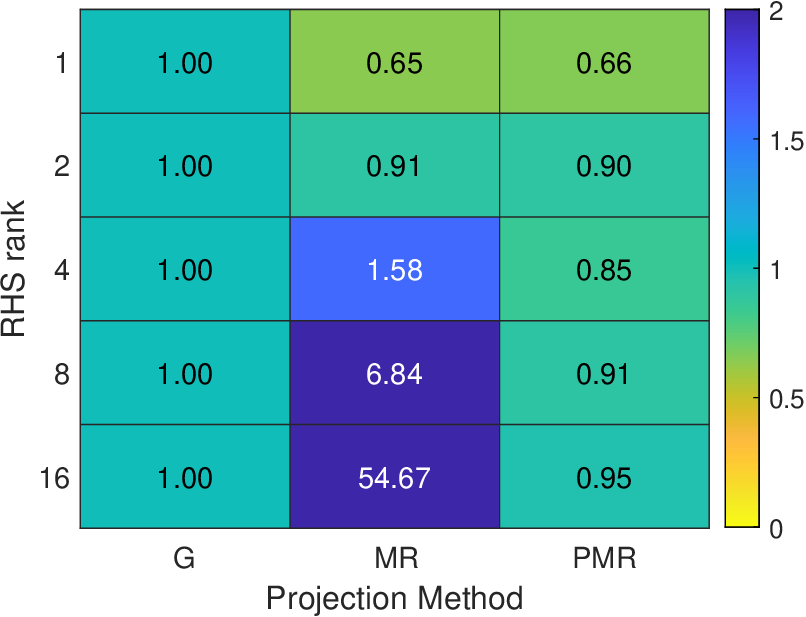} &
            \includegraphics[width=.45\textwidth]{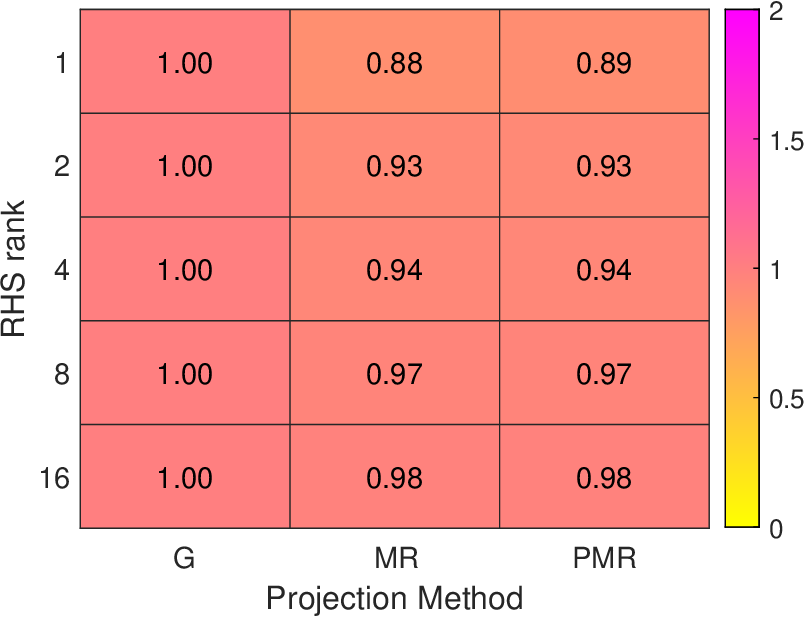}
        \end{tabular}
    \end{center}
    \caption{Performance results for Example~\ref{ex:vary_r_laplacian_2d_rand_n100}. \label{fig:vary_r_laplacian_2d_rand_n100}}
\end{figure}

\begin{example} \label{ex:vary_r_conv_diff_3d_rand_n25}
    For this \convdiff problem we have fixed $\eps = 10^{-2}$.  The results are similar as for Example~\ref{ex:vary_r_laplacian_2d_rand_n100}, with both MR and PMR out-performing the Galerkin approach for $r = 1$.  However, as $r$ increases, MR becomes slower.  Meanwhile, PMR continues to demonstrate a slight advantage over Galerkin, with the advantage decreasing as $r$ increases.
\end{example}
\begin{figure}[htbp!]
    \begin{center}
        \small
        \begin{tabular}{cc}
            Timings ratio & Iteration counts ratio \\
            \includegraphics[width=.45\textwidth]{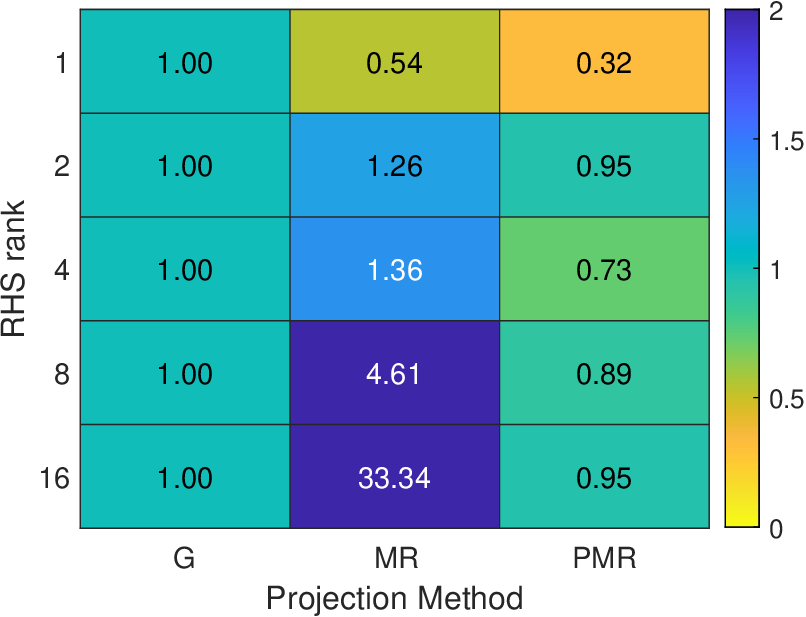} &
            \includegraphics[width=.45\textwidth]{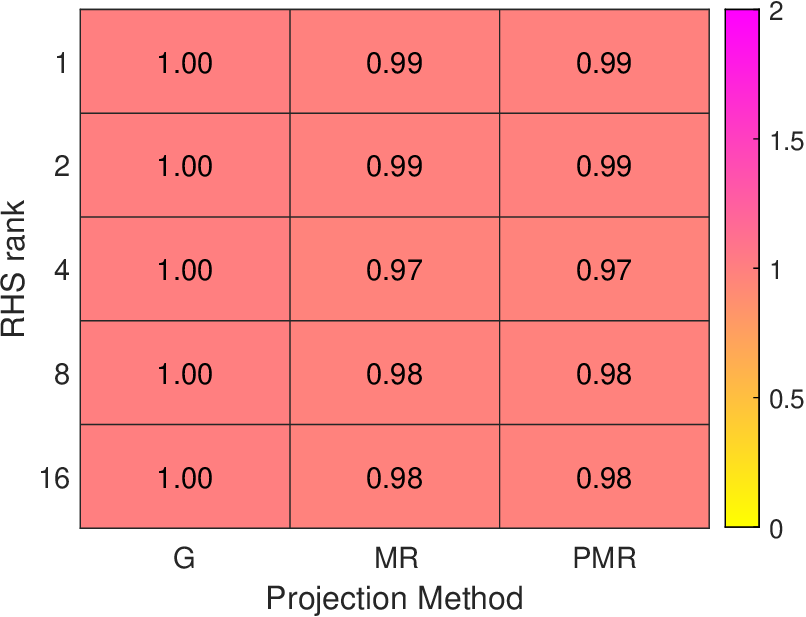}
        \end{tabular}
    \end{center}
    \caption{Performance results for Example~\ref{ex:vary_r_conv_diff_3d_rand_n25} \label{fig:vary_r_conv_diff_3d_rand_n25}}
\end{figure}

\subsection{Restarting} \label{sec:cr_num}
For the final set of examples, we study the behavior of the compress-and-restart strategy and compare between the Galerkin and PMR approaches.  We no longer consider the MR approach, as it is ill-suited for a compress-and-restart strategy, due to lacking the appropriate type of low-rank-modification formulation.

Our primary goal in this section is to study restarts and their effect on the monotonicity of the residual; we therefore set the compression tolerance to machine epsilon (i.e., approximately $10^{-16}$ for IEEE \texttt{int64}) to effectively turn off compression.

\begin{example} \label{ex:restarts_bad_cond_diag_rand_n100_r3}
    We again consider the \badconddiag problem with $n = 100$ and $r = 3$ and residual plot in Figure~\ref{fig:restarts_bad_cond_diag_rand}.  The maximum number of columns to be stored is set to $100$.  Although PMR only improves over the Galerkin approach by a few iterations, its monotonic residual makes it more reliable in practice.
\end{example}
\begin{figure}[htbp!]
    \begin{center}
        \includegraphics[width=.45\textwidth]{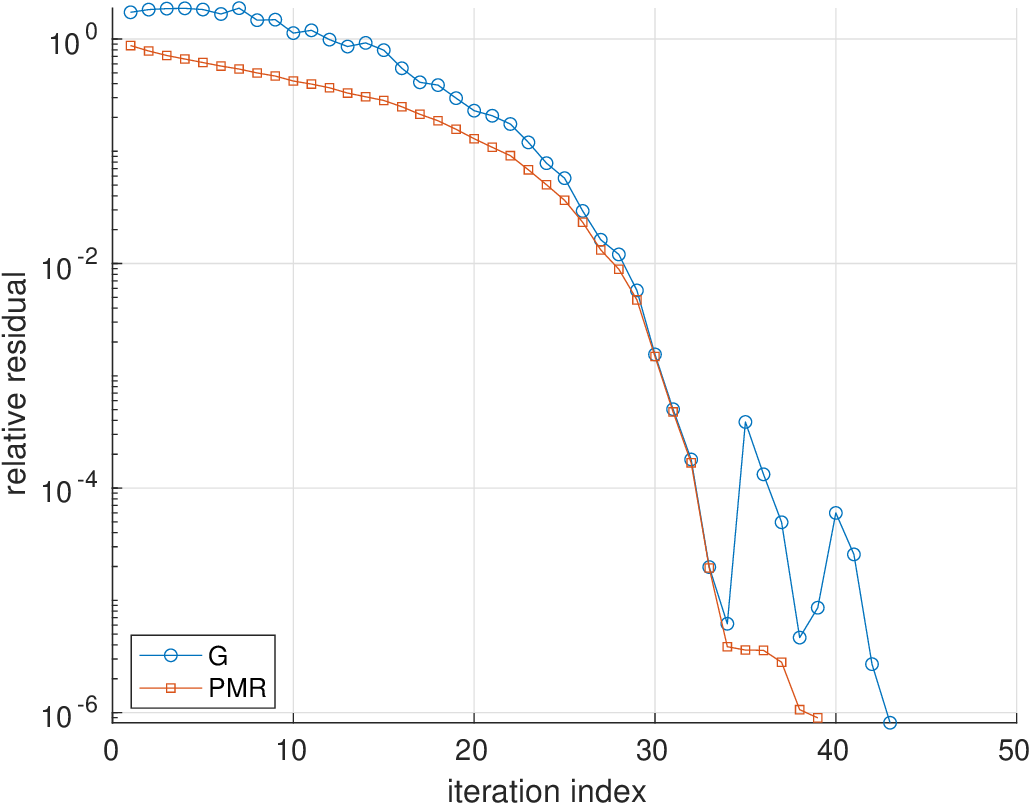}
    \end{center}
    \caption{Convergence results for Example~\ref{ex:restarts_bad_cond_diag_rand_n100_r3} \label{fig:restarts_bad_cond_diag_rand}}
\end{figure}

\begin{example} \label{ex:restarts_laplacian_2d_rand_n100_r3}
    We also look at the \laplacian once again, with $n = 10^4$ and $r = 3$ and maximum number of columns set to $96$.  Results are shown in Figure~\ref{fig:restarts_laplacian_2d_rand}.  Although the convergence of PMR is almost perfectly monotonic, it is notably slower than than of the Galerkin method.  It may therefore be reasonable to combine and switch methods after a certain point-- perhaps after the residual hits $10^{-4}$-- but determining a reliable heuristic in practice remains an open challenge.
\end{example}
\begin{figure}[htbp!]
    \begin{center}
        \includegraphics[width=.45\textwidth]{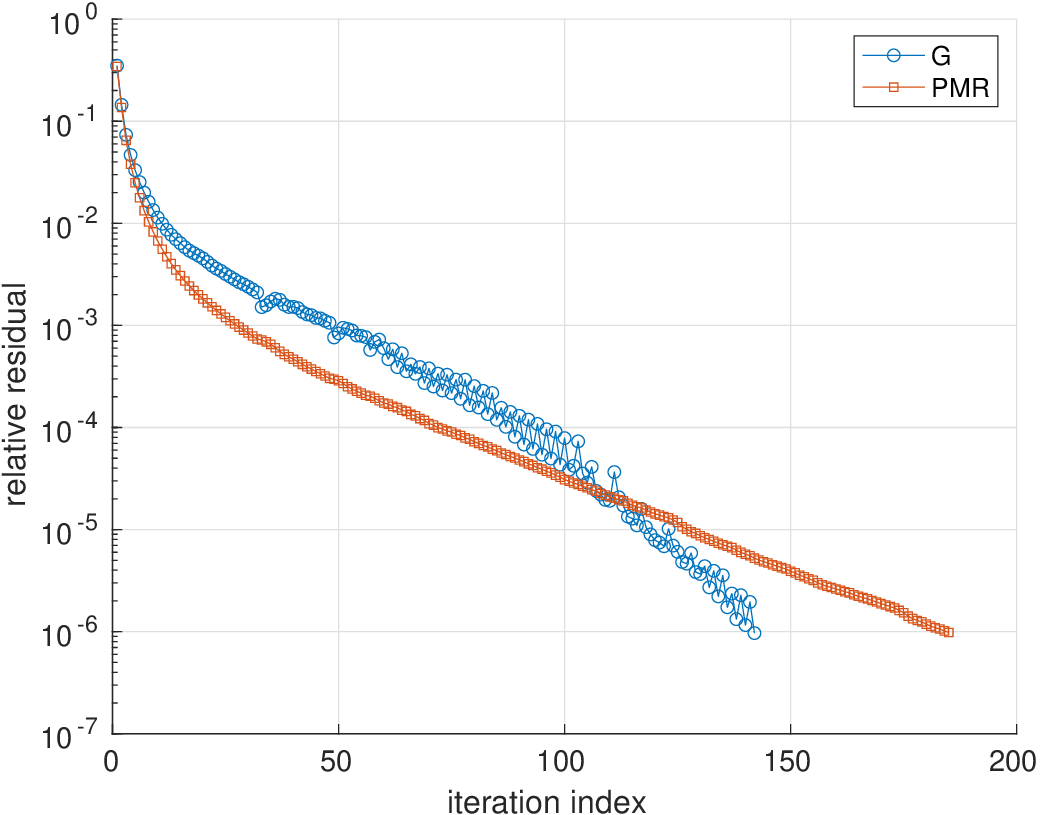}
    \end{center}
    \caption{Convergence results for Example~\ref{ex:restarts_laplacian_2d_rand_n100_r3}. \label{fig:restarts_laplacian_2d_rand}}
\end{figure}

\begin{example} \label{ex:restarts_rail_1357_n1357_r7}
    Finally, we return to the \rail problem, whose results for different maximum memory limits are shown in Figure~\ref{fig:restarts_rail_1357_n1357_r7}.  For the lower memory limit on the left, we see a situation similar to that of Example~\ref{ex:restarts_laplacian_2d_rand_n100_r3}, whereby the PMR residual is smooth and monotonic but ultimately converges more slowly than that of the Galerkin approach.  On the right, we see how the behavior changes with a higher memory tolerance: the PMR method is able to improve slightly and remain monotonic.
\end{example}
\begin{figure}[htbp!]
    \begin{center}
        \begin{tabular}{cc}
            $\texttt{memmax} = 420$ & $\texttt{memmax} = 840$ \\
            \includegraphics[width=.45\textwidth]{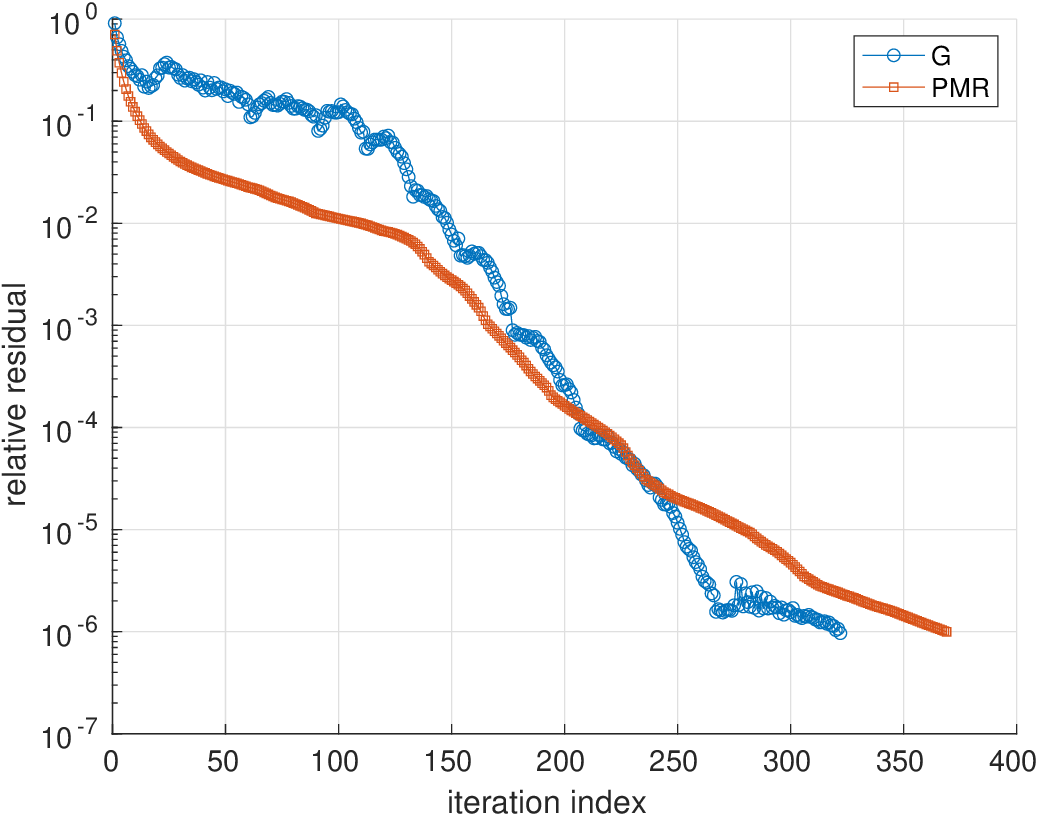} &
            \includegraphics[width=.45\textwidth]{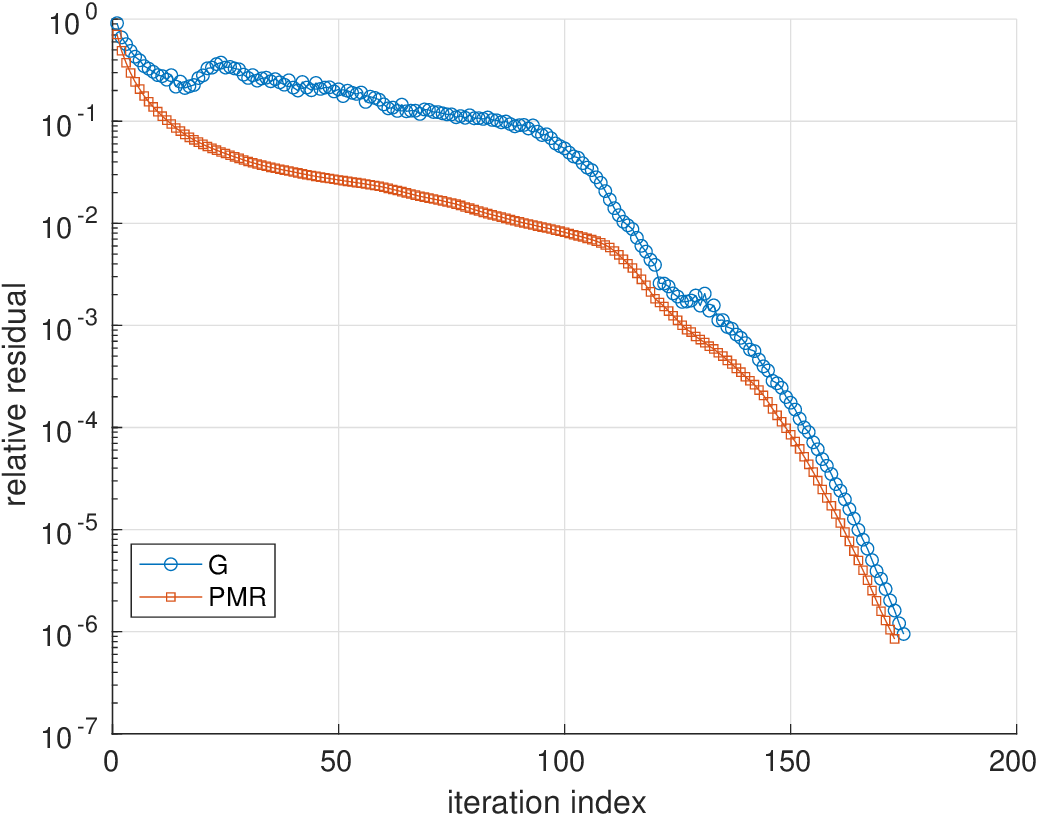}
        \end{tabular}
    \end{center}
    \caption{Convergence results for Example~\ref{ex:restarts_rail_1357_n1357_r7} \label{fig:restarts_rail_1357_n1357_r7}}
\end{figure}
\section{Conclusions and outlook} \label{sec:conclusions}
In this paper we have presented a new general framework for solving large-scale Lyapunov matrix equations. Low-rank-modified Galerkin methods, which include the Galerkin method as a special case with the zero modification, introduces a suitable low-rank correction in the projected equation aimed at achieving a certain target behaviour.  Driven by the goal of designing a computationally affordable scheme able to show a convergence rate similar to a minimal residual approach, we have proposed two non-trivial options for such a low-rank correction. 

The first one, PMR, is defined by taking inspiration from the relation between FOM and GMRES in the linear system setting. We showed that, under certain hypotheses on $A$, adopting this correction leads to a projection method that is well-defined. Moreover, we have depicted possible scenarios where we expect the performance of our new solver to be close to that achieved by MR. Our numerical results have confirmed such findings. On the other hand, further analysis is needed to fully understand the relation between the two approaches and other situations where PMR may be applicable.

The second low-rank correction we have proposed, NKS, is computed by minimizing the current residual norm at each iteration. In spite of its appealing theoretical features, computing this low-rank correction directly is unreasonably expensive, making its use limited in practice. Designing ad-hoc, more efficient optimization procedures is a venue worth pursuing to make this approach affordable in terms of computational cost. To this end, sketch-and-solve methods like, e.g., the Blendenpik algorithm~\cite{AvrMT10} could be a valid option. 

Our new framework is sufficiently flexible to handle low-rank corrections different from the ones we have proposed, and it would also be useful in designing low-rank corrections targeting other goals as, e.g., a certain spectral distribution of $\HH_m+\MM$.  We have also showed that, thanks to the low-rank format of the residual matrix computed by our novel low-rank modified Galerkin method, the latter can be successfully integrated in a compress-and-restart scheme for matrix equations.  

For the sake of simplicity, we have restricted our analysis to the use of polynomial Krylov subspaces. However, generalizing our approach to the case of more sophisticated approximation spaces like extended and rational Krylov subspaces is just a technical exercise. 

As a possible outlook, we envision our new low-rank modified Galerkin framework to be applied to the solution of other matrix equations like, e.g., generalized Lyapunov equations \kl{or algebraic Ricatti equations}.

\section*{Acknowledgments}%
\addcontentsline{toc}{section}{Acknowledgments}
This work began in 2020, during which time the first author was funded by the Charles University PRIMUS grant, project no.~PRIMUS/19/SCI/11.

The second author is member of the INdAM Research Group GNCS that partially supported this work through the funded project GNCS2023 ``Metodi avanzati per la risoluzione di PDEs su griglie strutturate, e non'' (CUP\_E53C22001930001).


\addcontentsline{toc}{section}{References}
\bibliographystyle{abbrvurl}
\bibliography{LowRank4Lyap}

\begin{thebibliography}{10}

\bibitem{Ant05b}
A.~C. Antoulas.
\newblock {\em Approximation of {{Large-Scale Dynamical Systems}}}.
\newblock {Society for Industrial and Applied Mathematics}, 2005.
\newblock \href {https://doi.org/10.1137/1.9780898718713}
  {\path{doi:10.1137/1.9780898718713}}.

\bibitem{AvrMT10}
H.~Avron, P.~Maymounkov, and S.~Toledo.
\newblock Blendenpik: {{Supercharging LAPACK}}'s {{Least-Squares Solver}}.
\newblock {\em SIAM J. Sci. Comput.}, 32(3):1217--1236, 2010.
\newblock \href {https://doi.org/10.1137/090767911}
  {\path{doi:10.1137/090767911}}.

\bibitem{BakES15}
J.~Baker, M.~Embree, and J.~Sabino.
\newblock Fast singular value decay for {{Lyapunov}} solutions with nonnormal
  coefficients.
\newblock {\em SIAM J. Matrix Anal. Appl.}, 36(2):656--668, 2015.
\newblock \href {https://doi.org/10.1137/140993867}
  {\path{doi:10.1137/140993867}}.

\bibitem{Bau08a}
U.~Baur.
\newblock Low rank solution of data-sparse {{Sylvester}} equations.
\newblock {\em Numer. Linear Algebra Appl.}, 15(9):837--851, 2008.
\newblock \href {https://doi.org/10.1002/nla.605} {\path{doi:10.1002/nla.605}}.

\bibitem{BauB06}
U.~Baur and P.~Benner.
\newblock Factorized solution of {{Lyapunov}} equations based on hierarchical
  matrix arithmetic.
\newblock {\em Computing}, 78(3):211--234, 2006.
\newblock \href {https://doi.org/10.1007/s00607-006-0178-y}
  {\path{doi:10.1007/s00607-006-0178-y}}.

\bibitem{BenKS16a}
P.~Benner, P.~K{\"u}rschner, and J.~Saak.
\newblock Frequency-limited balanced truncation with low-rank approximations.
\newblock {\em SIAM J. Sci. Comput.}, 38(1):A471--A499, 2016.
\newblock \href {https://doi.org/10.1137/15M1030911}
  {\path{doi:10.1137/15M1030911}}.

\bibitem{BenS13a}
P.~Benner and J.~Saak.
\newblock Numerical solution of large and sparse continuous time algebraic
  matrix {{Riccati}} and {{Lyapunov}} equations: a state of the art survey.
\newblock {\em GAMM-Mitt.}, 36(1):32--52, 2013.
\newblock \href {https://doi.org/10.1002/gamm.201310003}
  {\path{doi:10.1002/gamm.201310003}}.

\bibitem{ChaV05}
Y.~Chahlaoui and P.~Van~Dooren.
\newblock Benchmark {{Examples}} for {{Model Reduction}} of {{Linear
  Time-Invariant Dynamical Systems}}.
\newblock In P.~Benner, D.~C. Sorensen, and V.~Mehrmann, editors, {\em Dimens.
  {{Reduct}}. {{Large-Scale Syst}}.}, Lecture {{Notes}} in {{Computational
  Science}} and {{Engineering}}, pages 379--392, {Berlin, Heidelberg}, 2005.
  {Springer}.
\newblock \href {https://doi.org/10.1007/3-540-27909-1_24}
  {\path{doi:10.1007/3-540-27909-1_24}}.

\bibitem{DavH11}
T.~A. Davis and Y.~Hu.
\newblock The {{University}} of {{Florida Sparse Matrix Collection}}.
\newblock {\em ACM Trans. Math. Softw.}, 38(1), 2011.
\newblock \href {https://doi.org/10.1145/2049662.2049663}
  {\path{doi:10.1145/2049662.2049663}}.

\bibitem{DruS11}
V.~Druskin and V.~Simoncini.
\newblock Adaptive rational {{Krylov}} subspaces for large-scale dynamical
  systems.
\newblock {\em Systems Control Lett.}, 60(8):546--560, 2011.
\newblock \href {https://doi.org/10.1016/j.sysconle.2011.04.013}
  {\path{doi:10.1016/j.sysconle.2011.04.013}}.

\bibitem{FroGS14a}
A.~Frommer, S.~G{\"u}ttel, and M.~Schweitzer.
\newblock Convergence of restarted {{Krylov}} subspace methods for
  {{Stieltjes}} functions of matrices.
\newblock {\em SIAM J. Matrix Anal. Appl.}, 35(4):1602--1624, 2014.
\newblock \href {https://doi.org/10.1137/140973463}
  {\path{doi:10.1137/140973463}}.

\bibitem{FroLS17}
A.~Frommer, K.~Lund, and D.~B. Szyld.
\newblock Block {{Krylov}} subspace methods for functions of matrices.
\newblock {\em Electron. Trans. Numer. Anal.}, 47:100--126, 2017.
\newblock URL: \url{https://epub.oeaw.ac.at/0xc1aa5576%200x0037106a.pdf}.

\bibitem{FroLS20}
A.~Frommer, K.~Lund, and D.~B. Szyld.
\newblock Block {{Krylov}} subspace methods for functions of matrices {{II}}:
  {{Modified}} block {{FOM}}.
\newblock {\em SIAM J. Matrix Anal. Appl.}, 41(2):804--837, 2020.
\newblock \href {https://doi.org/10.1137/19M1255847}
  {\path{doi:10.1137/19M1255847}}.

\bibitem{GawJ90}
W.~Gawronski and J.-N. Juang.
\newblock Model reduction in limited time and frequency intervals.
\newblock {\em Int. J. Syst. Sci.}, 21(2):349--376, 1990.
\newblock \href {https://doi.org/10.1080/00207729008910366}
  {\path{doi:10.1080/00207729008910366}}.

\bibitem{Gut07}
M.~H. Gutknecht.
\newblock Block {{Krylov}} space methods for linear systems with multiple
  right-hand sides: {{An}} introduction.
\newblock In A.~H. Siddiqi, I.~S. Duff, and O.~Christensen, editors, {\em Mod.
  {{Math}}. {{Model}}. {{Methods Algorithms Real World Syst}}.}, pages
  420--447, {New Delhi}, 2007. {Anamaya}.
\newblock URL: \url{https://people.math.ethz.ch/~mhg/pub/delhipap.pdf}.

\bibitem{HesS52}
M.~R. Hestenes and E.~Stiefel.
\newblock Methods of conjugate gradients for solving linear systems.
\newblock {\em J. Res. Natl. Bur. Stand.}, 49(6):409--436, 1952.
\newblock \href {https://doi.org/10.6028/jres.049.044}
  {\path{doi:10.6028/jres.049.044}}.

\bibitem{Hig08a}
N.~J. Higham.
\newblock {\em Functions of {{Matrices}}}.
\newblock {SIAM}, {Philadelphia}, 2008.

\bibitem{HorJ91}
R.~A. Horn and C.~R. Johnson.
\newblock {\em Topics in {{Matrix Analysis}}}.
\newblock {Cambridge University Press}, {Cambridge}, 1991.

\bibitem{HuR92}
D.~Y. Hu and L.~Reichel.
\newblock Krylov-subspace methods for the {{Sylvester}} equation.
\newblock {\em Linear Algebra Appl.}, 172:283--313, 1992.
\newblock \href {https://doi.org/10.1016/0024-3795(92)90031-5}
  {\path{doi:10.1016/0024-3795(92)90031-5}}.

\bibitem{JbiR19}
K.~Jbilou and M.~Raydan.
\newblock Nonlinear {{Least-Squares Approach}} for {{Large-Scale Algebraic
  Riccati Equations}}.
\newblock {\em SIAM J. Sci. Comput.}, 41(4):A2193--A2211, 2019.
\newblock \href {https://doi.org/10.1137/18M1198922}
  {\path{doi:10.1137/18M1198922}}.

\bibitem{KorR05}
J.~G. Korvink and E.~B. Rudnyi.
\newblock Oberwolfach {{Benchmark Collection}}.
\newblock In P.~Benner, V.~Mehrmann, and D.~C. Sorensen, editors, {\em
  Dimension {{Reduction}} of {{Large-Scale Systems}}. {{Lecture Notes}} in
  {{Computational Science}} and {{Engineering}}}, volume~45, pages 311--315.
  {Springer}, {Berlin, Heidelberg}, 2005.
\newblock \href {https://doi.org/10.1007/3-540-27909-1_11}
  {\path{doi:10.1007/3-540-27909-1_11}}.

\bibitem{Kre14}
D.~Kressner.
\newblock Bivariate matrix functions.
\newblock {\em Oper. Matrices}, 8(2):449--466, 2014.
\newblock \href {https://doi.org/10.7153/oam-08-23}
  {\path{doi:10.7153/oam-08-23}}.

\bibitem{Kre19}
D.~Kressner.
\newblock A {{Krylov}} subspace method for the approximation of bivariate
  matrix functions.
\newblock In D.~A. Bini, F.~Di~Benedetto, E.~Tyrtyshnikov, and M.~Van~Barel,
  editors, {\em Struct. {{Matrices Numer}}. {{Linear Algebr}}.}, pages
  197--214. {Springer International Publishing}, 2019.
\newblock \href {https://doi.org/10.1007/978-3-030-04088-8_10}
  {\path{doi:10.1007/978-3-030-04088-8_10}}.

\bibitem{KreLMetal21}
D.~Kressner, K.~Lund, S.~Massei, and D.~Palitta.
\newblock Compress-and-restart block {{Krylov}} subspace methods for
  {{Sylvester}} matrix equations.
\newblock {\em Numer Linear Algebr Appl}, 28(1):e2339, 2021.
\newblock \href {https://doi.org/10.1002/nla.2339}
  {\path{doi:10.1002/nla.2339}}.

\bibitem{Kur16}
P.~K{\"u}rschner.
\newblock {\em Efficient low-rank solution of large-scale matrix equations}.
\newblock PhD thesis, Faculty of Mathematics, Otto-von-Guericke-University,
  Magdeburg, 2016.
\newblock URL:
  \url{https://pure.mpg.de/rest/items/item_2246796/component/file_2296741/content}.

\bibitem{LiW02}
J.-R. Li and J.~White.
\newblock Low {{Rank Solution}} of {{Lyapunov Equations}}.
\newblock {\em SIAM J. Matrix Anal. Appl.}, 24(1):260--280, 2002.
\newblock \href {https://doi.org/10.1137/S0895479801384937}
  {\path{doi:10.1137/S0895479801384937}}.

\bibitem{LinS13}
Y.~Lin and V.~Simoncini.
\newblock Minimal residual methods for large scale {{Lyapunov}} equations.
\newblock {\em Appl. Numer. Math.}, 72:52--71, 2013.
\newblock \href {https://doi.org/10.1016/j.apnum.2013.04.004}
  {\path{doi:10.1016/j.apnum.2013.04.004}}.

\bibitem{PalSS23a}
D.~Palitta, M.~Schweitzer, and V.~Simoncini.
\newblock Sketched and {{Truncated Polynomial Krylov Subspace Methods}}:
  {{Matrix Equations}}.
\newblock Technical Report arXiv:2311.16019, {arXiv}, 2023.
\newblock \href {https://doi.org/10.48550/arXiv.2311.16019}
  {\path{doi:10.48550/arXiv.2311.16019}}.

\bibitem{PalS16}
D.~Palitta and V.~Simoncini.
\newblock Matrix-equation-based strategies for convection-diffusion equations.
\newblock {\em BIT}, 56(2):751--776, 2016.
\newblock \href {https://doi.org/10.1007/s10543-015-0575-8}
  {\path{doi:10.1007/s10543-015-0575-8}}.

\bibitem{PalS20}
D.~Palitta and V.~Simoncini.
\newblock Optimality {{Properties}} of {{Galerkin}} and
  {{Petrov}}\textendash{{Galerkin Methods}} for {{Linear Matrix Equations}}.
\newblock {\em Vietnam J. Math.}, 48:791--807, 2020.
\newblock \href {https://doi.org/10.1007/s10013-020-00390-7}
  {\path{doi:10.1007/s10013-020-00390-7}}.

\bibitem{Pen00}
T.~Penzl.
\newblock Eigenvalue decay bounds for solutions of {{Lyapunov}} equations: the
  symmetric case.
\newblock {\em Syst. Control Lett}, 40(2):139--144, 2000.
\newblock \href {https://doi.org/10.1016/S0167-6911(00)00010-4}
  {\path{doi:10.1016/S0167-6911(00)00010-4}}.

\bibitem{Sim07}
V.~Simoncini.
\newblock A new iterative method for solving large-scale {{Lyapunov}} matrix
  equations.
\newblock {\em SIAM J. Sci. Comput.}, 29(3):1268--1288, 2007.
\newblock \href {https://doi.org/10.1137/06066120X}
  {\path{doi:10.1137/06066120X}}.

\bibitem{Sim16a}
V.~Simoncini.
\newblock Computational methods for linear matrix equations.
\newblock {\em SIAM Rev.}, 38(3):377--441, 2016.
\newblock \href {https://doi.org/10.1137/130912839}
  {\path{doi:10.1137/130912839}}.

\bibitem{SimG96}
V.~Simoncini and E.~Gallopoulos.
\newblock Convergence properties of block {{GMRES}} and matrix polynomials.
\newblock {\em Linear Algebra Appl.}, 247:97--119, 1996.
\newblock \href {https://doi.org/10.1016/0024-3795(95)00093-3}
  {\path{doi:10.1016/0024-3795(95)00093-3}}.

\bibitem{SnyZ70}
J.~Snyders and M.~Zakai.
\newblock On nonnegative solutions of the equation \${{AD}}+{{DA}} =-{{C}}\$.
\newblock {\em SIAM J. Appl. Math.}, 18:704--714, 1970.
\newblock \href {https://doi.org/10.1137/0118063} {\path{doi:10.1137/0118063}}.

\bibitem{VanV10}
B.~Vandereycken and S.~Vandewalle.
\newblock A {{Riemannian Optimization Approach}} for {{Computing Low-Rank
  Solutions}} of {{Lyapunov Equations}}.
\newblock {\em SIAM J. Matrix Anal. Appl.}, 31(5):2553--2579, 2010.
\newblock \href {https://doi.org/10.1137/090764566}
  {\path{doi:10.1137/090764566}}.

\bibitem{ZhoDG96}
K.~Zhou, J.~C. Doyle, and K.~Glover.
\newblock {\em Robust and {{Optimal Control}}}.
\newblock {Prentice-Hall}, {Upper Saddle River, NJ}, 1996.

\end{thebibliography}
  
\end{document}